\documentclass[11pt]{amsart}
\usepackage{amssymb, amsmath, amsthm}
\usepackage[latin1]{inputenc}
\usepackage{graphicx}
\usepackage[final]{hyperref}
\usepackage{color}
\usepackage{subfigure}
\usepackage{tikz}
\usetikzlibrary{positioning}
\usetikzlibrary{patterns}
\usetikzlibrary{shapes.geometric}
\usetikzlibrary{shapes.misc}

\usepackage[a4paper, centering]{geometry}
\usepackage{color}
\usepackage{graphicx}

\newtheorem{theorem}{Theorem}
\newtheorem{proposition}[theorem]{Proposition}
\newtheorem{lemma}[theorem]{Lemma}
\newtheorem{corollary}[theorem]{Corollary}
\theoremstyle{definition}
\newtheorem{definition}[theorem]{Definition}
\theoremstyle{remark}
\newtheorem{remark}[theorem]{Remark}

\def\R{\mathbb{R}}
\def\C{\mathbb{C}}
\def\N{\mathbb{N}}
\def\K{\mathcal{K}}

\def\P{\mathcal{P}}
\def\C{\mathcal{C}}

\def\inte{\rm int}
\def\ext{\rm ext}

\begin{document}

\title[]{On the honeycomb conjecture \\ for a class of minimal convex partitions }

\author[D. Bucur, I.~Fragal\`a, B. Velichkov, G. Verzini]{Dorin Bucur, Ilaria Fragal\`a, Bozhidar Velichkov, Gianmaria Verzini}

\address[Dorin Bucur]{
Institut Universitaire de France \\
Laboratoire de Math\'ematiques UMR 5127 \\
Universit\'e de Savoie,  Campus Scientifique \\
73376 Le-Bourget-Du-Lac (France)
}
\email{dorin.bucur@univ-savoie.fr}

\address[Ilaria Fragal\`a]{
Dipartimento di Matematica \\ Politecnico  di Milano \\
Piazza Leonardo da Vinci, 32 \\
20133 Milano (Italy)
}
\email{ilaria.fragala@polimi.it}

\address[Bozhidar Velichkov]{Laboratoire Jean Kuntzmann (LJK), Universit\'e Grenoble Alpes
\newline \indent
B\^atiment IMAG, 700 Avenue Centrale, 38401 Saint-Martin-d'H\`eres (France)}
\email{bozhidar.velichkov@univ-grenoble-alpes.fr}

\address[Gianmaria Verzini]{
Dipartimento di Matematica \\ Politecnico  di Milano \\
Piazza Leonardo da Vinci, 32 \\
20133 Milano (Italy)
}
\email{gianmaria.verzini@polimi.it}


\date{\today}

\begin{abstract}
We prove that the planar hexagonal honeycomb is asymptotically optimal for a large class of optimal partition problems, in which the cells are assumed to be convex, and the criterion is to minimize either the sum or the maximum among the energies of the cells, the cost being a shape functional which satisfies a few assumptions. They are: monotonicity under inclusions;  homogeneity under dilations;
a Faber-Krahn inequality for convex hexagons; a  convexity-type inequality for the map which associates with every $n \in \N$ the minimizers of $F$ among convex $n$-gons with given area.
In particular, our result allows to obtain the honeycomb conjecture for the Cheeger constant and for the logarithmic capacity (still assuming the cells to be convex). Moreover we show that, in order to get the conjecture also for the first Dirichlet eigenvalue of the Laplacian, it is sufficient to establish some facts about the behaviour of $\lambda _1$ among convex pentagons, hexagons, and heptagons with prescribed area.   \end{abstract}
\maketitle


\section{Introduction }

Given an open bounded subset $\Omega$ of $\R ^2$ with a Lipschitz boundary, we consider the problem of finding an optimal partition
$\{ E _1, \dots, E _k \}$ of $\Omega$ into $k$ convex cells, the energy being either of additive or of supremal type, {\it i.e.},
$$\sum _{ i = 1} ^ k F ( E _i) \qquad \text { or } \qquad \max _{i = 1, \dots, k } \big \{ F ( E _i  ) \big \}\,.$$
The cost functional $F$ is assumed to be homogeneous under dilations, and monotone under domain inclusion
on the class of convex bodies  in
$\R ^2$.

In case $F$ is monotone decreasing, the admissible configurations $\{ E _1, \dots, E _k \}$ are  {\it convex $k$-clusters} of $\Omega$, denoted by ${\C} _k (\Omega)$  and meant as families
of $k$ convex bodies which are contained into $\Omega$ and have mutually disjoint interiors.
So our problems read
\begin{eqnarray}
\displaystyle m _{k} (\Omega) = \inf  \Big \{  \sum _{i = 1} ^ k F (E _i) \ :\  \{E_i\} \in \C _k (\Omega) \Big \} \, , \quad  \
& \label{f:pb0}
  \\  \noalign{\medskip}
\displaystyle M _{k} (\Omega) = \inf  \Big \{ \max_{i = 1, \dots, k}  F (E _i) \ :\  \{E_i\} \in \C _k (\Omega) \Big \} \, .\quad  \
& \label{f:pb}
\end{eqnarray}
In case $F$ is monotone increasing, to make the minimization nontrivial (namely the infimum nonzero) we have to consider as admissible configurations
$\{ E _1, \dots, E _k \}$ only the convex $k$-clusters which, loosely speaking, cover the whole of $\Omega$: they are
called
{\it convex $k$-partitions}  of $\Omega$, and are denoted by ${\P} _k (\Omega)$.
Moreover,  for increasing functionals
we just consider problems of the supremal type
\begin{equation}\label{f:pb2}
\displaystyle M _{k} (\Omega) = \inf  \Big \{ \max_{i = 1, \dots, k}  F (P _i\cap \Omega) \ :\  \{P_i\} \in {\P _k (\Omega)} \Big \} \, .
\end{equation}
Actually, for $F$ increasing,
the counterpart of  problem \eqref{f:pb0} would become meaningful only after adding an isoperimetric constraint on the cells
which would change dramatically the nature of the problem
({\it cf.}\ the well-known case solved by Hales in the celebrated paper \cite{Hales},
when the cost is the total perimeter of the partition, see also \cite{Morgan}).

We are interested in studying for which kind of variational energies $F$ optimal partition problems of the kind \eqref{f:pb0}, \eqref{f:pb}, or \eqref{f:pb2}
satisfies the ``honeycomb conjecture''. Roughly, it can be stated as
the fact that,  in the limit for $k$ very large, an optimal packing will be made of  translations of an identical shape, given precisely by a regular hexagon $H$.
A simple mathematical formulation can be given as the asymptotic law
\begin{equation}\label{f:rough}
\lim _{ k \to + \infty } \frac{1}{k ^ \gamma} \Big ( \frac{|\Omega|}{|H| } \Big ) ^ \gamma  m _k (\Omega)  =  F (H)\,,
\end{equation}
and similarly with $m _k (\Omega)$ replaced by $M _k (\Omega)$. Clearly,
the exponent $\gamma>0$ appearing in \eqref{f:rough} depends on the homogeneity degree of the energy $F$ under domain dilations.

The original motivation of our work is a conjecture formulated by Caffarelli and Lin in \cite{CaffLin}. It predicts the validity of
\eqref{f:rough} when $F$ is the first Dirichlet eigenvalue $\lambda _1$ of the Laplacian, namely
when $m _k (\Omega)$ is given  by
\begin{equation}\label{f:pb22}
\displaystyle m _{k} (\Omega) = \inf  \Big \{ \sum_{i = 1, \dots, k}  \lambda _1 (E _i) \ :\  E _i \subseteq \Omega \,, \  |E_i|  \in (0,  + \infty)\, , \ |E _i \cap E _j | = 0
 \Big \} \, .
\end{equation}
As a matter of fact,  optimal spectral partitions have received an increasing attention in the last decade,
including also the case when the energy of the partition is the maximal eigenvalue among the chambers (in particular by Helffer and coauthors);
without any attempt of completeness,  let us quote the papers \cite{BoVe, BNHV, BuBuHe, CTV03, CTV05, He07, He10, HeHoTe, Ramos}. In particular, in \cite{HeHoTe}, a similar honeycomb conjecture for the maximal eigenvalue problem is attributed to Van den Berg.

We emphasize that  the substantial difference between problem \eqref{f:pb22} considered by Caffarelli and Lin and our problem \eqref{f:pb0} is that  we added the quite stringent constraint that the cells are, {\it a priori},  convex. Clearly, this  yields a great simplification, and in this sense the present work is conceived as a first step towards the complete study of the problem without the convexity constraint.
This kind of approach  was inspired by the fact that, for perimeter minimizing partitions made by convex polygons, the proof of the honeycomb conjecture is much simpler, and indeed it was given by Fejes T\'oth some decades before Hales' breakthrough (see \cite{FT64}).

As a counterpart, we take the freedom to work with very general shape functionals, by imposing on $F$ very few assumptions.

 Our main results are stated in Section \ref{main} below:
 Theorems \ref{t:honeycomb0}, \ref{t:honeycomb}, and \ref{t:honeycomb2}
deal respectively with problems \eqref{f:pb0}, \eqref{f:pb} and \eqref{f:pb2}, under suitable hypotheses on a generic shape functional $F$.
 As a consequence, the honeycomb conjecture is obtained for the Cheeger constant  and for logarithmic capacity (see Corollaries \ref{p:h0}-\ref{p:h} and \ref{p:LC});
 moreover, we  provide some relatively simple sufficient conditions for its validity also to the case of the first Dirichlet eigenvalue
 (see Proposition \ref{p:l1}).

 The proofs of Theorem \ref{t:honeycomb0}, Theorem \ref{t:honeycomb}, Theorem \ref{t:honeycomb2}, and Proposition \ref{p:l1}  are given in the subsequent sections.
 The Appendix contains some analytical computations used to obtain Corollaries \ref{p:h0}-\ref{p:h} and \ref{p:LC}.

 \section{Main results}\label{main}

Throughout the  paper, unless otherwise specified, we will use the following notation:
\begin{itemize}
\item $\Omega$ is an open bounded subset of $\R ^2$ with a Lipschitz boundary;
\item $H$ denotes the unit area regular hexagon;
\item $\mathcal K ^ 2$ is the family of convex bodies  in $\R ^2$ ({\it i.e.}, convex compact sets in the plane having a  nonempty interior).
\end{itemize}
We next introduce the class of convex $k$-clusters.
\begin{definition} We denote by ${\C} _k (\Omega)$ the class of {\it convex $k$-clusters} of $\Omega$, meant as families
$\{ E _i \} _{\{i = 1, \dots, k \}}$
 of subsets of $\R^2$ such that:

 \smallskip
 \begin{itemize}
 \item[{-}] $E_i \in \mathcal K ^2$ for every $i$;

 \smallskip
 \item[{-}] $E_i \subseteq \Omega$ for every $i$;

\smallskip
\item[{-}] $|E _i \cap E _j | = 0$ for every $i\neq j$.
 \end{itemize}
\end{definition}

Our main results on the asymptotic behaviour of optimal partition problems for decreasing functionals read as follow.
We give two distinct statements for the case of additive and supremal energies because the assumptions we need in the two cases are slightly different
from each other (see Remark  \ref{r:FK}, which collects our comments on the theorems stated hereafter).

\begin{theorem}\label{t:honeycomb0}  Assume that $F:\K ^ 2 \to [0, + \infty)$ satisfies the following conditions:
\begin{itemize}
\item[(H1)] Domain monotonicity:
$$\Omega _ 1 \subseteq \Omega _ 2 \Rightarrow F (\Omega _ 1 ) \geq F (\Omega _2)\,.$$
\item[(H2)] Homogeneity:
$$\exists\, \alpha >0 \ :\ F ( t \Omega) = t ^ {- \alpha} F (\Omega)\qquad \hbox{ for every } t >0\,.$$
 \item[(H3)] Behaviour on polygons: setting
 $$\gamma (n) := \min \Big \{ F (P) |P| ^ {\alpha/2} \ :\ P \  \text{n-gon in }     \mathcal K ^ 2  \} \qquad \forall n \in \N \, , $$
we have

\medskip
\begin{itemize}
\item[(i)] $\displaystyle \gamma (6) = F (H)$;
 \smallskip
\item[(ii)] $\displaystyle \frac{1}{k} \sum _{i=1} ^ k n _i \leq 6 \ \Rightarrow \ \frac{1}{k}  \sum _{i=1} ^ k \gamma (n _i ) ^ {2/(\alpha+2)} \geq  \gamma (6 ) ^ {2/(\alpha+2)}$.

 \end{itemize}
 \end{itemize}

 \smallskip
Then, in the limit as $k \to + \infty$, the optimal partition problem
$$
\displaystyle m _{k} (\Omega) = \inf  \Big \{ \sum_{i = 1}^  k F (E _i) \ :\  \{E_i\} \in \C _k (\Omega) \Big \}
$$
is solved by a packing of regular hexagons, namely it holds
\begin{equation}\label{f:conj0}
\lim_{k \to + \infty}\frac{|\Omega| ^ {\alpha /2}} {k ^ {(\alpha+2)/2}}  m _k (\Omega)  =
F ( H) \,.
\end{equation}
\end{theorem}

\medskip

\begin{theorem}\label{t:honeycomb}  Assume that $F:\K ^ 2 \to [0, + \infty)$ satisfies the following conditions:
\begin{itemize}
\item[(H1)] Domain monotonicity:
$$\Omega _ 1 \subseteq \Omega _ 2 \Rightarrow F (\Omega _ 1 ) \geq F (\Omega _2)\,.$$
\item[(H2)] Homogeneity:
$$\exists\, \alpha >0 \ :\ F ( t \Omega) = t ^ {- \alpha} F (\Omega)\qquad \hbox{ for every } t >0\,.$$
 \item[(H3)] Behaviour on polygons: setting
 $$\gamma (n) := \min \Big \{ F (P) |P| ^ {\alpha/2} \ :\ P \  \text{n-gon in }     \mathcal K ^ 2  \} \qquad \forall n \in \N \, , $$
we have

\medskip
\begin{itemize}
\item[(i)] $\displaystyle \gamma (6) = F (H)$;
 \smallskip
\item[(ii)] $\displaystyle \frac{1}{k} \sum _{i=1} ^ k n _i \leq 6 \ \Rightarrow \ \frac{1}{k}  \sum _{i=1} ^ k \gamma (n _i ) ^ {2/\alpha} \geq  \gamma (6 ) ^ {2/\alpha}$.

 \end{itemize}
 \end{itemize}

 \smallskip
Then, in the limit as $k \to + \infty$, the optimal partition problem
$$
\displaystyle M _{k} (\Omega) = \inf  \Big \{ \max_{i = 1, \dots, k}  F (E _i) \ :\  \{E_i\} \in \C _k (\Omega) \Big \}
$$
is solved by a packing of regular hexagons, namely it holds
\begin{equation}\label{f:conj}
\lim_{k \to + \infty}\frac{|\Omega| ^ {\alpha /2}} {k ^ {\alpha/2}}  M _k (\Omega)  =
F ( H) \,.
\end{equation}
\end{theorem}

\bigskip
\begin{remark}\label{r:FK}
\begin{itemize}
\item[(i)] Notice that assumption (H3) (ii)  in Theorem \ref{t:honeycomb}  is less stringent that the same assumption in  Theorem \ref{t:honeycomb0}, because
the exponent $2/\alpha$   appearing therein is strictly larger than its corresponding one $2/(\alpha +2)$ .

\smallskip

\item[(ii)]
If the functional $F$ satisfies a discrete Faber-Krahn inequality with regular polygons as optimal domains,
we can give an easier to handle sufficient condition for the validity of
hypothesis (H3). Namely  assume that, for every $n \in \N$, the minimum $\gamma(n)$ is achieved for  the  regular $n$-gon with unit area $P _n ^*$, that is $\gamma (n) = F ( P _n ^*)$; this can be equivalently stated in the form of an inequality, to which we refer as the discrete Faber-Krahn inequality:
$$F(P_n)\le F(P_n^*),\ \text{for every n-gon}\  P_n\in\mathcal K ^ 2\ \text{of unit area}.$$

\noindent Then, in order to check (H3), it is enough to show that the following condition is satisfied
(with $\beta = 2/(\alpha +2)$ or $\beta = 2/\alpha$ respectively in case of Theorem \ref{t:honeycomb0} or Theorem \ref{t:honeycomb}):

\bigskip
\begin{itemize}
\item[(H3)']
the map
$n \mapsto F ( P _n ^*) ^ {\beta}$  admits a decreasing and convex extension $\varphi$ on $[3, +\infty)$.
\end{itemize}

\bigskip
\noindent Indeed, part (i) of hypothesis (H3) is clearly true since the Faber-Krahn inequality for $n=6$ yields $\gamma (6) = F ( P _ 6 ^*) = F (H)$, whereas part (ii)   follows from (H3)' and the Jensen's inequality:
$$\frac{1}{k}  \sum _{i=1} ^ k \gamma (n _i ) ^ {\beta}
= \frac{1}{k}  \sum _{i=1} ^ k \varphi (n_i )
 \geq \varphi \Big ({\frac{1}{k} \sum _{i=1} ^ k n _i}  \Big ) \geq \varphi (6) = \gamma (6)   ^{\beta} \,.   $$

\item[(iii)] If in addition the discrete Faber-Krahn inequality holds in a quantitative way (see Remark \ref{rem:quantitativeFK}),  it is reasonable to expect that, beyond \eqref{f:conj0}-\eqref{f:conj}, one can obtain the convergence of any optimal partition to a hexagonal tiling.
\end{itemize}
\end{remark}

\bigskip
We are now going to state a dual counterpart of Theorem \ref{t:honeycomb} for increasing functionals.

In this case,   we have to work with convex partitions which saturate the domain $\Omega$,
otherwise the infimum of the optimal partition problem would become zero (simply by taking $k$ disjoint balls of infinitesimal radius
contained into $\Omega$).
Thus we introduce the following

\begin{definition} {We denote by $\P _k (\Omega)$ the class of {\it convex $k$-partitions} of $\Omega$, meant as families
$\{ P _i \} _{\{i = 1, \dots, k \}}$
 of subsets of $\R ^2$ such that:}

 \smallskip
 \begin{itemize}
 \item[{-}] $P_i$ is a polygon in $\mathcal K ^ 2$ for every $i$;

 \smallskip
 \item[{-}] {$\bigcup _i \big (P _ i \cap \overline \Omega \big ) = \overline \Omega$;}

 \smallskip
 \item[{-}] $|P _i \cap P _j | = 0$ for every $i\neq j$.
 \end{itemize}
\end{definition}

\begin{theorem}\label{t:honeycomb2}  Let $\Omega$ be an open bounded convex subset of $\R ^2$.
Assume that $F:\K ^ 2 \to [0, + \infty)$ satisfies the following conditions:
\begin{itemize}
\item[(H1)] Domain monotonicity:
$$\Omega _ 1 \subseteq \Omega _ 2 \Rightarrow F (\Omega _ 1 ) \leq F (\Omega _2)\,. $$
\item[(H2)] Homogeneity:
$$\exists\, \alpha >0 \ :\ F ( t \Omega) = t ^ { \alpha} F (\Omega)\qquad \hbox{ for every } t >0\,.$$
 \item[(H3)] Behaviour on polygons: setting
 $$\gamma (n) := \min \Big \{ F (P) |P| ^ {-\alpha/2} \ :\ P \  \text{n-gon in }     \mathcal K ^ 2  \} \qquad \forall n \in \N \, , $$

\medskip
\begin{itemize}
\item[(i)] $\displaystyle \gamma (6) = F (H)$;
 \smallskip
\item[(ii)] {it is possible to extend $\gamma$ to a function defined on $[3, + \infty)$ which is continuous at the point $6$ and is such that, for some $k _0 \in \N$ and $\delta _0 >0$,  }

\medskip

 \end{itemize}
 $\displaystyle \frac{1}{k} \sum _{i=1} ^ k n _i \leq 6  + \delta\, ,  \text{ with  } k \geq k _0 \text{ and } 0<\delta \leq \delta _0 \ \Rightarrow \ \frac{1}{k}  \sum _{i=1} ^ k \gamma (n _i ) ^ {-2/\alpha} \leq  \gamma  \big (6  + \delta \big ) ^ {-2/\alpha}$.

 \end{itemize}

 \smallskip
Then, in the limit as $k \to + \infty$, the optimal partition problem
$$
m _{k} (\Omega) = \inf  \Big \{ \max_{i = 1, \dots, k}  F (P _i \cap \Omega) \ :\  \{P_i\} \in \P _k (\Omega) \Big \}
$$
is solved by a packing of regular hexagons, namely it holds
\begin{equation}\label{f:conj2}
\lim_{k \to + \infty} \frac{ k ^ {\alpha/2} } { |\Omega| ^ {\alpha /2} }  {m _k (\Omega)}  =
F ( H) \,.
\end{equation}
\end{theorem}

\begin{remark}\label{r:FK2}
Notice that condition (H3) (ii) in Theorem \ref{t:honeycomb2} is a little bit more involved than the corresponding one in Theorem \ref{t:honeycomb}. The reason is that, when passing from decreasing to increasing functionals, we have to deal with convex $k$-partitions rather than convex $k$-clusters,  and consequently a key argument in the proof needs to be adapted and refined (see Remark \ref{r:proof} for more details).
However we emphasize that, when $F$ satisfies a discrete Faber-Krahn inequality  stating that
 $\gamma (n) = F ( P _n ^*)$  (with $P _n ^*$ the  regular $n$-gon with unit area), one can still formulate a simpler sufficient condition for the validity of assumption (H3) in Theorem \ref{t:honeycomb2}, which in this case reads
\smallskip
\begin{itemize}
\item[(H3)']
the map
$n \mapsto F ( P _n ^*) ^ {-2/\alpha}$ admits an increasing and concave extension $\varphi$ on $[3, + \infty)$.
\end{itemize}
This is readily checked by arguing as in Remark \ref{r:FK}.
\end{remark}

\begin{remark}
We point out  that Theorem \ref{t:honeycomb2} applies in particular to the shape functional
$F(\Omega)  = {\rm Per}(\Omega)$. Indeed, assumptions (H1) and (H2) are fulfilled because perimeter  is monotone increasing on convex sets under domain inclusion, and positively homogeneous of degree $\alpha = 1$. Moreover, since the regular $n$-gon minimizes perimeter among $n$-gons with given area, for the validity of (H3) it is enough to check condition (H3)', which is immediate by using the formula
$${\rm Per }  ( P _n ^*) = 2 \sqrt { n \tan \Big ( \frac{\pi}{n} \Big ) }\,.$$

\end{remark}

\bigskip
We present now the application of Theorems \ref{t:honeycomb0}, \ref{t:honeycomb} and \ref{t:honeycomb2} to some relevant examples of shape functionals $F$ of variational type.

\subsection {The Cheeger constant}

Let us recall that the  Cheeger constant of $\Omega$ is defined by
\begin{equation}\label{f:defh}
h (\Omega):=\inf \left \{ \frac{{\rm Per} (E, \R ^2)}{|E|}\ :\ E \hbox{ measurable}\, , \ {E\subseteq   \Omega} \right \}\,,
\end{equation}
where ${\rm Per}(  E, \R ^2)$ denotes the perimeter of $E$ in the sense of De Giorgi.
The minimization problem (\ref{f:defh}) is  named after J. Cheeger, who introduced it in \cite{Ch} and proved the inequality
$\lambda _ 1 (\Omega) \geq  ( h (\Omega) /2) ^ 2 $.
In recent years the Cheeger constant has attracted an increasing attention: we address the interested reader to the review papers \cite{Leo, Pa} and to the numerous references therein.
Let us also mention that optimal partition problems for the Cheeger constant have been recently considered in \cite{Car15},
under the form \eqref{f:pb22}, and with the aim of finding bounds for the asymptotics of the same problem for the first Dirichlet eigenvalue.\\

We claim that Theorems \ref{t:honeycomb0} and \ref{t:honeycomb} apply to $F (\Omega): =h (\Omega)$.

Indeed, it is immediate from its definition that $h (\Omega)$ satisfies the monotonicity assumption (H1) and the homogeneity assumption (H2) (with $\alpha = 1$).

Concerning assumption (H3) (i), in view of  Remark \ref{r:FK} (ii), we recall that
the regular $n$-gon $P _n ^*$ of unit area minimizes the Cheeger constant $h$ among all polygons of the same area and same number of sides (even without the convexity constraint). This result has been recently proved in \cite{BF16}  (incidentally, in the light of Remark \ref{r:FK} (iii), let us also mention that a
quantitative version of such result has appeared in \cite{CarQuantitative}).

In order to check (H3) (ii), it is enough to show that the map $n \mapsto h( P _n ^*)^{ \beta}$
 admits a decreasing convex extension on $[3, + \infty)$,
where the exponent $\beta$ equals $2/3$ and $2$ in case respectively of
 Theorems \ref{t:honeycomb0} and \ref{t:honeycomb}.
This is readily done in view of the explicit expression of $ h( P _n ^*)$, which reads (see for instance \cite{BF16} or \cite{KaLR}):
$$ h( P _n ^*) = \frac{2n\sin(\pi/n)+\sqrt{2\pi n\sin(2\pi/n)}}{\sqrt{2 n\sin(2\pi/n)}}\,.$$
For completeness, the computations are included in the Appendix, Lemma \ref{giamma0} and Lemma \ref{giamma1}. We reassume the results obtained in the discussion above in the following:

\begin{corollary}\label{p:h0}
Let
\begin{equation}
\displaystyle m _{k} (\Omega) = \inf  \Big \{ \sum_{i = 1}^ k  h (E _i) \ :\  \{E_i\} \in \C _k (\Omega) \Big \} \, .
\end{equation}
Then it holds
\begin{equation}
\lim_{k \to + \infty}\frac{|\Omega| ^ {1 /2}} {k ^ {3/2}}  m _k (\Omega)  =
h ( H) \,.
\end{equation}
\end{corollary}

\begin{corollary}\label{p:h}
Let
\begin{equation}
\displaystyle M _{k} (\Omega) = \inf  \Big \{ \max_{i = 1, \dots, k}  h (E _i) \ :\  \{E_i\} \in \C _k (\Omega) \Big \} \, .
\end{equation}
Then it holds
\begin{equation}
\lim_{k \to + \infty}\frac{|\Omega| ^ {1 /2}} {k ^ {1/2}}  M _k (\Omega)  =
h ( H) \,.
\end{equation}
\end{corollary}

\bigskip

\subsection{The first Dirichlet Laplacian eigenvalue.}

Let us now consider the case of the first Dirichlet eigenvalue of the Laplacian,  $F (\Omega) = \lambda _1 (\Omega)$.
We observe that the monotonicity and homogeneity assumptions (H1) and (H2) of Theorem \ref{t:honeycomb} are satisfied (with $\alpha = 2$).  On the other hand, the Faber-Krahn inequality for the fist Dirichlet eigenvalue of polygons is a long-standing open  problem
(for a discussion, see for instance \cite[Section 3.3]{H06}).  Thus we have to work directly on the validity of assumption (H3),
handling the function $\gamma (n)$ defined by
\begin{equation}\label{f:gl1}
\gamma (n) := \min  \Big \{ \lambda _ 1 (P) |P|  \ :\ P \  \text{n-gon in }     \mathcal K ^ 2  \Big \}\,.
\end{equation}

As a consequence of Theorems  \ref{t:honeycomb0} and \ref{t:honeycomb}, we can assert that the honeycomb conjecture for $\lambda _1$ holds true {\it provided} one has some piece of information about the behaviour of $\gamma (n)$ {\it just} for $n = 5,6,7$.
Indeed we have the following result:

\begin{proposition}\label{p:l1}
Let
\begin{eqnarray}
\displaystyle m _{k} (\Omega) = \inf  \Big \{ \sum_{i = 1}^ k  \lambda _1 (E _i) \ :\  \{E_i\} \in \C _k (\Omega) \Big \}
&
\\
\displaystyle M _{k} (\Omega) = \inf  \Big \{ \max_{i = 1, \dots, k}  \lambda _1 (E _i) \ :\  \{E_i\} \in \C _k (\Omega) \Big \}
&
\, ,
\end{eqnarray}
and assume that the map $\gamma (n)$ defined by \eqref{f:gl1} satisfies
\begin{eqnarray}
\gamma (6) = \lambda _ 1 (H) \qquad \qquad \qquad & \label{i1}
\\ \noalign{\medskip}
\gamma (5) \geq a:=  6,022 \pi  \ \ \text{  and } \ \ \gamma (7) \geq b:= 5.82 \pi & \label{i2}
\end{eqnarray}
\medskip
Then it holds
\begin{eqnarray}
\lim_{k \to + \infty}\frac{|\Omega|} {k^2 }   m _k (\Omega)  =
\lambda _1 ( H) \,. &
\\
\lim_{k \to + \infty}\frac{|\Omega|} {k }   M _k (\Omega)  =
\lambda _1 ( H) \,. &
 \end{eqnarray}

\end{proposition}

\begin{remark} Let us emphasize that the assumptions made on
$\gamma (5)$ and $\gamma (7)$ seem much easier to check than to
the assumption $\gamma (6) = F (H)$. Namely, while the latter corresponds exactly to prove Faber-Krahn inequality for the principal frequency of convex hexagons (and as such is quite challenging), for $\gamma (5)$ and $\gamma (7)$ we just ask estimates from below, which are likely more at hand through a numerical proof.  \end{remark}

\subsection{The logarithmic capacity.}
Let us recall that the logarithmic capacity of $\Omega$ is defined by
$$
{\rm LogCap} (\Omega) := \exp \Big[ { - \lim \limits _{|x| \to + \infty} ( u (x) - \ln |x| )\Big ] } \,,
$$
where $u$ is the equilibrium potential of $\Omega$, namely the unique solution to the Dirichlet boundary value problem
$$\begin{cases}
\Delta u = 0 & \text{ in } \R ^ 2 \setminus \overline \Omega
\\ \noalign{\smallskip}
u = 0 & \text{ on } \partial \Omega
\\ \noalign{\smallskip}
u (x) \sim \ln |x| & \text{ as } |x| \to + \infty\,.
\end{cases}
$$
It can also be identified with the conformal radius or  with the transfinite diameter of $\Omega$, for more details see  for instance \cite{Landkof, Cocu}.

We claim that Theorem \ref{t:honeycomb2} applies to $F (\Omega): ={\rm LogCap} (\Omega)$.
  Indeed, it is immediate from its definition that ${\rm LogCap} (\Omega)$ satisfies assumptions (H1) and (H2) (with $\alpha = 1$).
Moreover,  Solynin and Zalgaller  have proved in \cite{SoZa} that the regular $n$-gon with unit area $P _n ^*$ minimizes ${\rm LogCap} (\Omega)$ among polygons with the same area and number of sides (even without the convexity constraint).
Then, recalling Remark \ref{r:FK2}, in order to check assumption (H3)  it is sufficient to show that the  map
$n \mapsto {\rm LogCap} ( P _n ^*) ^{-2}$ admits an increasing and concave extension on $[3, + \infty)$.
This is readily done in view of the explicit expression of ${\rm LogCap} ( P _n ^*)$, which  reads (see \cite{SoZa})
$$
{\rm LogCap} ( P _n ^*)=\frac{\sqrt{n\,\tan\left(\frac{\pi}{n}\right)}\,\Gamma\left(1+\frac1n\right)}{\sqrt \pi\, 2^{2/n}\,\Gamma \left(\frac12+\frac1n\right)}\, ,
$$
where $\Gamma$ is the Euler Gamma function (see Lemma \ref{giamma2} in the Appendix).

Thus we have:

\begin{corollary}\label{p:LC}
Let
\begin{equation}
\displaystyle M _{k} (\Omega) = \inf  \Big \{ \max_{i = 1, \dots, k}  {\rm LogCap} (P _i \cap \Omega) \ :\  \{P_i\} \in \P_k (\Omega) \Big \} \, .
\end{equation}
Then it holds
\begin{equation}
\lim_{k \to + \infty}\frac{ |\Omega| ^ {1 /2} } {k ^ {1/2}}M _k (\Omega)  =
{\rm LogCap} ( H) \,.
\end{equation}
\end{corollary}

\section{Proof of Theorems  \ref{t:honeycomb0} and \ref{t:honeycomb} }\label{proof1}

The proofs of Theorems  \ref{t:honeycomb0} and \ref{t:honeycomb} are obtained by combining the next two lemmas.

\begin{definition} Let $(H_i ) _{i \in I}$ denote a tiling of $\R^2$ made by copies $H _i$ of the unit area regular hexagon $H$.
By {\it $k$-hexagonal structure}, we mean a connected set obtained as the union of $k$ hexagons lying in the family $(H _i)$.
\end{definition}

\begin{lemma}\label{t:helffer}   Let
$\Omega _k$ denote a generic $k$-hexagonal structure.
\begin{itemize}
\item[(i)]
If $F$ satisfies the assumptions of Theorem \ref{t:honeycomb0}, and
\begin{equation}\label{f:cells0}m _{k} (\Omega_k) = k F  (H) \qquad \forall k \in \N\,, \end{equation}
then  the conclusion \eqref{f:conj0} of Theorem \ref{t:honeycomb0} holds true.
\item[(ii)]
If $F$ satisfies the assumptions of Theorem \ref{t:honeycomb}, and
\begin{equation}\label{f:cells}M_{k} (\Omega_k) = F  (H) \qquad \forall k \in \N\,, \end{equation}
then  the conclusion \eqref{f:conj} of Theorem \ref{t:honeycomb} holds true.
\end{itemize}
\end{lemma}

 \begin{lemma}\label{t:clusters}
 Let
$\Omega _k$ denote a generic $k$-hexagonal structure.
\begin{itemize}
\item[(i)] If $F$ satisfies the assumptions of Theorem \ref{t:honeycomb0}, then \eqref{f:cells0} holds true.
\smallskip
\item[(ii)] If $F$ satisfies the assumptions of Theorem \ref{t:honeycomb}, then \eqref{f:cells} holds true.
\end{itemize}
\end{lemma}

%

\medskip {\it Proof of Lemma \ref{t:helffer}}.
The argument of the proof is inspired from and quite close to the one from \cite[Section 4]{BNHV}. However for convenience of the reader we report the detailed proof below.

Let $(H_i ) _{i \in I}$ denote a tiling of $\R^2$ made by copies $H _i$ of the unit area regular hexagon $H$.

For any positive factor of dilation $\rho$,  we set $\rho \Omega := \{ \rho x \ :\ x \in \Omega \}$, and we introduce the following families of indices:

$$
\begin{array}{ll}
& I ^ {\inte} (\rho, \Omega)  := \Big \{ i \in I \ :\ H _ i \subset (\rho \Omega) \Big \} \,,
\\  \noalign{\bigskip}
& I ^ {\ext} (\rho, \Omega) := I ^ {\inte} (\rho,\Omega) \cup \Big \{ i \in I  \ :\ H _ i \cap \partial ( \rho \Omega) \neq \emptyset \Big \} \,.
\end{array}
$$
Then, for every $k \in \N$, we set
$$
\begin{array}{ll}
& \rho ^ {\inte} (k, \Omega)  := \inf \Big \{ \rho >0 \ : \ \sharp I ^ {\inte} (\rho, \Omega) \geq k \Big \} \,,
\\  \noalign{\bigskip}
& \rho ^ {\ext} (k, \Omega)  := \sup \Big \{ \rho >0 \ :\ \sharp I ^ {\ext} (\rho, \Omega) \leq k \Big \} \,.
\end{array}
$$
Since
$$\sharp I ^ {\inte} (\rho, \Omega) \sim \sharp  I ^ {\ext} (\rho, \Omega) \sim \frac{|\Omega|}{|H|} \rho ^ 2 =|\Omega|\rho ^ 2 \qquad \hbox{ as } \rho \to + \infty\, ,$$
it holds
\begin{equation}\label{f:ae} \lim _{ k \to + \infty} \frac{  \rho ^ {\inte} (k, \Omega)  }{\sqrt k} = \lim _{ k \to + \infty} \frac{  \rho ^ {\ext} (k, \Omega)  }{\sqrt k}  = \frac{\sqrt {|H|} }{\sqrt{ |\Omega| }  } = \frac{1}{\sqrt{ |\Omega| } } \,.
\end{equation}

We observe that, since $F$ satisfies assumption (H2),  the map $\Omega \mapsto m _k (\Omega)$ is
homogeneous of degree $- \alpha$ under dilations. Moreover, it is monotone decreasing under inclusions, as
$$\Omega \subseteq \Omega' \quad \Rightarrow \quad \C _k (\Omega) \subseteq \C _k (\Omega ' )\,.$$

Now we proceed  to prove separately  statements (i) and (ii), though the two cases are quite similar to each other.

\medskip
{\it Proof of statement (i)}.  We deduce the statement by combining an upper bound and a lower bound for $m _k (\Omega)$.

\medskip
{\it -- Upper bound for $m _k (\Omega)$}.
We take $\rho = \rho ^ {\inte} (k, \Omega)$.
By definition of $m_k(\cdot)$ it holds $m _k (\rho \Omega) \leq k  F (H)$.
Then, by using  the homogeneity of $m _k (\cdot)$, we get
$$m_k (\Omega) = \rho^ \alpha m _k ( \rho \Omega) \leq \rho^ \alpha k F (H)\,.$$
We infer that
\begin{equation}\label{f:ub0} \limsup _{k \to + \infty} \frac {m_k (\Omega)} {k ^ {(\alpha+2)/2}} \leq \limsup _{k \to + \infty}  \frac{\big( \rho^ {\inte}(k, \Omega)  \big ) ^ \alpha }  {k ^ {(\alpha+2)/2}}  k F (H) = \frac{1}{|\Omega| ^ {{\alpha}/{2}} }  F (H)  \,,
\end{equation}
where the last equality follows from \eqref{f:ae}.

\medskip
{\it -- Lower bound for $m _k (\Omega)$}.
We take $\rho = \rho ^ {\ext} (k, \Omega)$.  By using the homogeneity and  decreasing monotonicity of $m _k (\cdot)$, and the hypothesis \eqref{f:cells0}, we get
$$m _k ( \Omega) = \rho ^ \alpha m _k (\rho \Omega) \geq \rho ^ \alpha  m _{k }  (\Omega _{k } )
= \rho ^ \alpha  k F (H)
 \,.$$
We infer that
\begin{equation}\label{f:lbp0} \liminf _{k \to + \infty} \frac {m _k (\Omega)} {k ^ {(\alpha +2)/2}}
\geq \liminf _{k \to + \infty}  \frac{\big ( \rho^ {\ext}(k, \Omega)  \big ) ^ \alpha }  {k ^ {(\alpha+2)/2}} k F (H)=
\frac{1}{|\Omega| ^ {{\alpha}/{2}} }  F (H)  \,,
\end{equation}
where again the last equality follows from \eqref{f:ae}.

\smallskip
The proof is achieved by combining \eqref{f:ub0} and \eqref{f:lbp0}.

\medskip

{\it Proof of statement (ii)}.  Similarly as above, we deduce the statement by combining an upper bound and a lower bound for $M _k (\Omega)$.

\medskip
{\it -- Upper bound for $M _k (\Omega)$}.
We take $\rho = \rho ^ {\inte} (k, \Omega)$.
By definition of $M_k(\cdot)$ it holds $M _k (\rho \Omega) \leq F (H)$.
Then, by using  the homogeneity of $M _k (\cdot)$, we get
$$M_k (\Omega) = \rho^ \alpha M _k ( \rho \Omega) \leq \rho^ \alpha F (H)\,.$$
We infer that
\begin{equation}\label{f:ub} \limsup _{k \to + \infty} \frac {M_k (\Omega)} {k ^ {\alpha/2}} \leq \limsup _{k \to + \infty}  \frac{\big( \rho^ {\inte}(k, \Omega)  \big ) ^ \alpha }  {k ^ {\alpha/2}} F (H) = \frac{1}{|\Omega| ^ {{\alpha}/{2}} }  F (H)  \,,
\end{equation}
where the last equality follows from \eqref{f:ae}.

\medskip
{\it -- Lower bound for $M _k (\Omega)$}.
We take $\rho = \rho ^ {\ext} (k, \Omega)$.  By using the homogeneity and  decreasing monotonicity of $M _k (\cdot)$, and the hypothesis \eqref{f:cells}, we get
$$M _k ( \Omega) = \rho ^ \alpha M _k (\rho \Omega) \geq \rho ^ \alpha  M _{k }  (\Omega _{k } )
= \rho ^ \alpha F (H)
 \,.$$
We infer that
\begin{equation}\label{f:lbp} \liminf _{k \to + \infty} \frac {M _k (\Omega)} {k ^ {\alpha/2}}
\geq \liminf _{k \to + \infty}  \frac{\big ( \rho^ {\ext}(k, \Omega)  \big ) ^ \alpha }  {k ^ {\alpha/2}} F (H)=
\frac{1}{|\Omega| ^ {{\alpha}/{2}} }  F (H)  \,,
\end{equation}
where again the last equality follows from \eqref{f:ae}.


\smallskip
The proof is achieved by combining \eqref{f:ub} and \eqref{f:lbp}. \qed

\bigskip
In order to prove Lemma \ref{t:clusters}, we need some preliminaries.

\begin{definition} We say that a  $k$-hexagonal structure is a {\it $k$-triangle} if it has the shape of an equilateral  triangle
(see Figure \ref{fig-tiling} (a) below). In particular,  the number of cells $k$ is of the form $\displaystyle k=l(l+1)/2$, where $l$ is the number of cells on one of the sides of the triangle. The boundary of the $k$-triangle is composed of three sets $B_1$, $B_2$ and $B_3$, each of the sets being composed of these sides whose exterior normal has positive scalar product with the vector $\nu_1$, $\nu_2$ and $\nu_3$, respectively.

\end{definition}


\begin{figure}[ht]
\begin{center}
\hspace*{-0.5cm}
\subfigure[A 15-triangle.]{
\begin{tikzpicture}[scale=0.60, transform shape]
\node[regular polygon, regular polygon sides=6, minimum width=2cm,draw] (reg1) at (0,0){};
  \node[regular polygon, regular polygon sides=6, minimum width=2cm,draw] (reg1) at (2*sin{60}*sin{60},2*sin{60}*cos{60}){};
  \node[regular polygon, regular polygon sides=6, minimum width=2cm,draw] (reg1) at (2*sin{60}*sin{60},-2*sin{60}*cos{60}){};

  \node[regular polygon, regular polygon sides=6, minimum width=2cm,draw] (reg1) at (3,0){};
  \node[regular polygon, regular polygon sides=6, minimum width=2cm,draw] (reg1) at (4*sin{60}*sin{60},4*sin{60}*cos{60}){};
  \node[regular polygon, regular polygon sides=6, minimum width=2cm,draw] (reg1) at (4*sin{60}*sin{60},-4*sin{60}*cos{60}){};

  \node[regular polygon, regular polygon sides=6, minimum width=2cm,draw] (reg1) at (3+2*sin{60}*sin{60},2*sin{60}*cos{60}){};
  \node[regular polygon, regular polygon sides=6, minimum width=2cm,draw] (reg1) at (3+2*sin{60}*sin{60},-2*sin{60}*cos{60}){};

  \node[regular polygon, regular polygon sides=6, minimum width=2cm,draw] (reg1) at (6*sin{60}*sin{60},6*sin{60}*cos{60}){};
  \node[regular polygon, regular polygon sides=6, minimum width=2cm,draw] (reg1) at (6*sin{60}*sin{60},-6*sin{60}*cos{60}){};

  \node[regular polygon, regular polygon sides=6, minimum width=2cm,draw] (reg1) at (8*sin{60}*sin{60},8*sin{60}*cos{60}){};
  \node[regular polygon, regular polygon sides=6, minimum width=2cm,draw] (reg1) at (8*sin{60}*sin{60},-8*sin{60}*cos{60}){};

  \node[regular polygon, regular polygon sides=6, minimum width=2cm,draw] (reg1) at (3+4*sin{60}*sin{60},4*sin{60}*cos{60}){};
  \node[regular polygon, regular polygon sides=6, minimum width=2cm,draw] (reg1) at (3+4*sin{60}*sin{60},-4*sin{60}*cos{60}){};

  \node[regular polygon, regular polygon sides=6, minimum width=2cm,draw] (reg1) at (6,0){};

\draw[very thick,blue] (-1+0,0) -- (-1+cos{60},sin{60}) -- (-1+cos{60}+1,sin{60}) -- (-1+2*cos{60}+1,2*sin{60}) -- (-1+2*cos{60}+2,2*sin{60}) -- (-1+3*cos{60}+2,3*sin{60}) -- (-1+3*cos{60}+3,3*sin{60}) -- (-1+4*cos{60}+3,4*sin{60}) -- (-1+4*cos{60}+4,4*sin{60}) -- (-1+5*cos{60}+4,5*sin{60}) -- (-1+5*cos{60}+5,5*sin{60});
\draw[very thick,red] (-1+0,0) -- (-1+cos{60},-sin{60}) -- (-1+cos{60}+1,-sin{60}) -- (-1+2*cos{60}+1,-2*sin{60}) -- (-1+2*cos{60}+2,-2*sin{60}) -- (-1+3*cos{60}+2,-3*sin{60}) -- (-1+3*cos{60}+3,-3*sin{60}) -- (-1+4*cos{60}+3,-4*sin{60}) -- (-1+4*cos{60}+4,-4*sin{60}) -- (-1+5*cos{60}+4,-5*sin{60}) -- (-1+5*cos{60}+5,-5*sin{60});
\draw[very thick, green] (-1+5*cos{60}+5,5*sin{60}) -- (-1+6*cos{60}+5,4*sin{60}) -- (-1+5*cos{60}+5,3*sin{60}) -- (-1+6*cos{60}+5,2*sin{60}) -- (-1+5*cos{60}+5,1*sin{60}) -- (-1+6*cos{60}+5,0*sin{60}) -- (-1+5*cos{60}+5,-1*sin{60}) -- (-1+6*cos{60}+5,-2*sin{60}) -- (-1+5*cos{60}+5,-3*sin{60}) -- (-1+6*cos{60}+5,-4*sin{60}) -- (-1+5*cos{60}+5,-5*sin{60});
\draw[very thick, green, -latex] (7.5,0) -- (8.5,0) node[midway,above] {$\nu_3$};
\draw[very thick, blue, -latex] (-1+3*cos{60}+2+0.5*cos{120},3*sin{60}+0.5*sin{120}) -- (-1+3*cos{60}+2+1.5*cos{120},3*sin{60}+1.5*sin{120}) node[midway,above,sloped] {$\nu_1$};
\draw[very thick, red, -latex] (-1+3*cos{60}+0.5*cos{240}+2,-3*sin{60}+0.5*sin{240}) -- (-1+3*cos{60}+2+1.5*cos{240},-3*sin{60}+1.5*sin{240}) node[midway,above,sloped] {$\nu_2$};
\end{tikzpicture}}\hspace{0.5cm}
\subfigure[An example of convex polygon $P_i$ touching only one side of the $k$-triangle.]{
\begin{tikzpicture}[scale=0.60, transform shape]
  \node[regular polygon, regular polygon sides=6, minimum width=2cm,draw] (reg1) at (0,0){};
  \node[regular polygon, regular polygon sides=6, minimum width=2cm,draw] (reg1) at (2*sin{60}*sin{60},2*sin{60}*cos{60}){};
  \node[regular polygon, regular polygon sides=6, minimum width=2cm,draw] (reg1) at (2*sin{60}*sin{60},-2*sin{60}*cos{60}){};

  \node[regular polygon, regular polygon sides=6, minimum width=2cm,draw] (reg1) at (3,0){};
  \node[regular polygon, regular polygon sides=6, minimum width=2cm,draw] (reg1) at (4*sin{60}*sin{60},4*sin{60}*cos{60}){};

    \node[regular polygon, regular polygon sides=6, minimum width=2cm,draw] (reg1) at (3+2*sin{60}*sin{60},2*sin{60}*cos{60}){};

    \node[regular polygon, regular polygon sides=6, minimum width=2cm,draw] (reg1) at (0,-2*sin{60}){};

   \draw[very thick, pattern=north west lines, opacity=0.6] (sin{30}+sin{30}/0.7-1/0.7,0) -- (1,1.7*cos{30}) -- (1.7*cos{30}+sin{30}+0.7*1.7*cos{30}-0.7*1+0.7*sin{30},2*cos{30}+0.7*2*cos{30}-0.7*1.7*cos{30}) -- (3.3,-0.3) -- (1,-1) -- (sin{30}+sin{30}/0.7-1/0.7,0);
\node at (1.8,0.7) {$P_i$};

\draw[very thick,blue] (-1+0,0) -- (-1+cos{60},sin{60}) -- (-1+cos{60}+1,sin{60}) -- (-1+2*cos{60}+1,2*sin{60}) -- (-1+2*cos{60}+2,2*sin{60}) -- (-1+3*cos{60}+2,3*sin{60}) -- (-1+3*cos{60}+3,3*sin{60}) -- (-1+4*cos{60}+3,4*sin{60}) -- (-1+4*cos{60}+4,4*sin{60});
\end{tikzpicture}}
\hspace{0.5cm}
\subfigure[An example of convex polygon $P_i$ touching two sides of the $k$-triangle.]{
\begin{tikzpicture}[scale=0.60, transform shape]

  \node[regular polygon, regular polygon sides=6, minimum width=2cm,draw] (reg1) at (0,0){};
  \node[regular polygon, regular polygon sides=6, minimum width=2cm,draw] (reg1) at (2*sin{60}*sin{60},2*sin{60}*cos{60}){};
  \node[regular polygon, regular polygon sides=6, minimum width=2cm,draw] (reg1) at (2*sin{60}*sin{60},-2*sin{60}*cos{60}){};

  \node[regular polygon, regular polygon sides=6, minimum width=2cm,draw] (reg1) at (3,0){};
  \node[regular polygon, regular polygon sides=6, minimum width=2cm,draw] (reg1) at (4*sin{60}*sin{60},4*sin{60}*cos{60}){};
  \node[regular polygon, regular polygon sides=6, minimum width=2cm,draw] (reg1) at (4*sin{60}*sin{60},-4*sin{60}*cos{60}){};

  \node[regular polygon, regular polygon sides=6, minimum width=2cm,draw] (reg1) at (3+2*sin{60}*sin{60},2*sin{60}*cos{60}){};
  \node[regular polygon, regular polygon sides=6, minimum width=2cm,draw] (reg1) at (3+2*sin{60}*sin{60},-2*sin{60}*cos{60}){};

  \node[regular polygon, regular polygon sides=6, minimum width=2cm,draw] (reg1) at (6*sin{60}*sin{60},6*sin{60}*cos{60}){};
  \node[regular polygon, regular polygon sides=6, minimum width=2cm,draw] (reg1) at (6*sin{60}*sin{60},-6*sin{60}*cos{60}){};

 \draw[very thick, pattern=north west lines, opacity=0.6] (sin{30}+sin{30}/0.7-1/0.7,0) -- (1,1.7*cos{30}) -- (1.7*cos{30}+sin{30}+0.7*1.7*cos{30}-0.7*1+0.7*sin{30},2*cos{30}+0.7*2*cos{30}-0.7*1.7*cos{30}) -- (1.7*cos{30}+sin{30}+1.3*1.7*cos{30}-1.3*1+1.3*sin{30},-2*cos{30}-1.3*2*cos{30}+1.3*1.7*cos{30}) -- (1,-1.7*cos{30}) -- (sin{30}+sin{30}/0.7-1/0.7,0);
\node at (2,-0.5) {$P_i$};

\draw[very thick,blue] (-1+0,0) -- (-1+cos{60},sin{60}) -- (-1+cos{60}+1,sin{60}) -- (-1+2*cos{60}+1,2*sin{60}) -- (-1+2*cos{60}+2,2*sin{60}) -- (-1+3*cos{60}+2,3*sin{60}) -- (-1+3*cos{60}+3,3*sin{60}) -- (-1+4*cos{60}+3,4*sin{60}) -- (-1+4*cos{60}+4,4*sin{60});
\draw[very thick,red] (-1+0,0) -- (-1+cos{60},-sin{60}) -- (-1+cos{60}+1,-sin{60}) -- (-1+2*cos{60}+1,-2*sin{60}) -- (-1+2*cos{60}+2,-2*sin{60}) -- (-1+3*cos{60}+2,-3*sin{60}) -- (-1+3*cos{60}+3,-3*sin{60}) -- (-1+4*cos{60}+3,-4*sin{60}) -- (-1+4*cos{60}+4,-4*sin{60});
\end{tikzpicture}}
\end{center}
\label{fig-tiling}
\end{figure}

\begin{lemma}\label{FT}
Let $\Omega _k$ be a $k$-triangle, and let $\{ E_i \}$ be a convex $k$-cluster of $\Omega_k$.
Consider the optimization problem
$$\max  \Big \{\sum _{i = 1} ^ k |C_i| \ :\ \{C _i\} \text{ convex $k$-cluster of $\Omega_k$, $C _i \supseteq E_i \ \forall i = 1, \dots, k$} \Big \}\,. $$
Then:
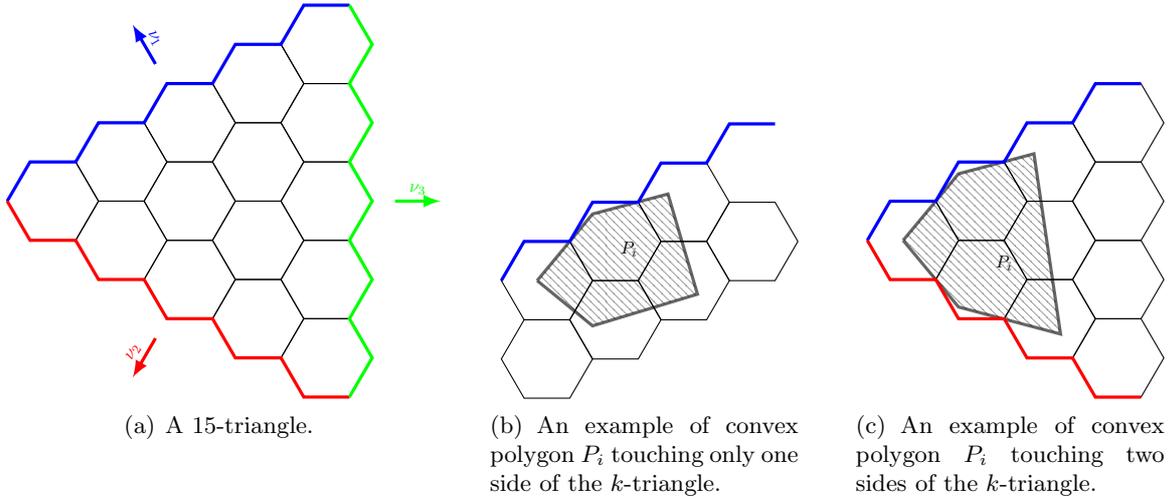
\begin{itemize}
\item[(i)] A solution $\{C _i^ {opt} \}$ exists, and each $C_i ^ {opt}$ is a convex polygon $P _i$ with a finite number of sides.
\medskip
\item[(ii)] Every side of $P_i$ intersects in its relative interior $\cup _{j \neq i } \partial P _j \cup \partial \Omega _k$.
\medskip
\item[(iii)] If $n _i$ denotes the number of sides of $P _i$, it holds
$$\frac{1}{k} \sum _{i = 1} ^ k n _i \leq 6\,.$$
\medskip
\end{itemize}
\end{lemma}

\proof
(i) The existence is straightforward since,  if we endow the class of convex $k$-clusters with the Hausdorff topology, we are maximizing a continuous functional on a compact set under the closed constraint $C_i \supseteq E _i$.

Since every $C _i^ {opt}$ is a convex set, to prove that it is a polygon it is enough to  show that the portions of $\partial C _i^ {opt}$ which are free, meaning that they do not lie neither on $\partial C _j^ {opt}$ for any $j \neq i$, nor on $\partial \Omega _k$, are line segments
(possibly degenerated into a point).
Let $\Gamma$ be any such portion, and assume by contradiction it is not a line segment. Then there are two distinct points $p, q$ on $\Gamma$ such that the tangent lines to $\Gamma$ through $p$ and $q$ are not parallel;  consequently, it is possible to construct a convex set, contained into $\Omega _k \setminus \bigcup _{j \neq i}  C _j^ {opt}$,  containing $E _i$, and having a strictly larger area than  $C _i ^{opt}$,  contradicting the maximality of $\sum _{i=1} ^k |C _i^ {opt}|$.

Notice finally that the number of sides of each  $P _i$  is necessarily finite in view of the convexity of the polygons, which ensure that the contact set $\partial P_i\cap\partial P_j$ of the polygons $P_i$ and $P _j$ is connected (and hence that the number of contacts is finite).

\smallskip
(ii) Assume by contradiction that there is a side $S$ of $P_i$ which does not intersect in its relative interior $\cup _{j \neq i } \partial P _j \cup \partial \Omega _k$. Then by adding to $P _i$ a small triangle with basis $S$, we would find a  convex set, contained into
$\Omega _k \setminus \bigcup _{j \neq i}  P _j$, containing $E_i$, and having a strictly larger area than $P _i$.
This contradicts the maximality of $\sum _{i=1} ^k |P _i|$.

\smallskip
(iii) Denote by $P _0$ the unbounded connected component of $\R ^2 \setminus \bigcup _{i= 1} ^ k P _i$.  In order to compute  the sum $\sum _{i = 1} ^ k n _i $, we construct a suitable planar graph associated with the family of polygons $\{P_0, \dots, P_k\}$.

For every $i = 0, 1, \dots, k$, we associate with the polygon $P _i$ a vertex.

Before we construct the edges of the graph, we make some preliminary observations:
\begin{itemize}
\item{} For $j$ and $i$ in $\{1, \dots, k\}$, the intersection $\partial P _i \cap \partial P _j$ may be either empty, or a segment  with strictly positive $\mathcal H ^ 1$ measure, or a point;
\item{} For $j=0$ and $i$ in $\{1, \dots, k\}$, the intersection $\partial P _0 \cap \partial P _i$ may be either empty, or a set  with strictly positive $\mathcal H ^ 1$ measure, or a finite number of points. Notice in particular that $\partial P _0 \cap \partial P _i$  may be disconnected.
\end{itemize}

 Then we give the following rule:
 \begin{itemize}
 \item{} for $j$ and $i$ in $\{1, \dots, k\}$, we connect the vertices corresponding to
 $P _j$ and $P _i$ with one edge if and only if $\partial P _i \cap \partial P _j$ is either a line segment with a strictly positive $\mathcal H ^ 1$ measure, or a point lying in the interior of a side of $P _i$ or $P _j$;
 \item{} for $j=0$ and $i$ in $\{1, \dots, k\}$, we connect the vertices corresponding to $P _0$ and $P _i$ if and only if $\partial P _i \cap \partial P_0$ contains a side of $P_i$ (that is a side that has an interior contact point with $\partial \Omega_k$, but not with any of the convex polygons $P_j$, for $j\ge1$ and $j\neq i$).  In this case, we connect the vertices corresponding to $P_0$ and $P_i$ exactly with one edge for each connected component of $\partial P _i \cap \partial P_0$ containing a side of $P_i$.
 \end{itemize}

Note that it is possible to perform the above construction in such a way that different edges do not intersect outside the vertices, so that the graph thus constructed is planar (in order to construct a representation of the graph in the plane it is sufficient to associate to each polygon $P_i$ a point $X_i$ in its interior and then connect the points $X_i$ and $X_j$ by a curve passing through $\partial P_i\cap\partial P_j$).

We notice that this planar graph may admit multiple edges connecting $P_0$ to one of the convex polygons $P_i$. On the other hand, by construction, we have that each face of the graph has at least three edges. In fact if two edges $e_1$ and $e_2$, both connecting $P_i$ to $P_0$, determine one face of the graph, then the intersections of these edges with $\partial P_i$ are on the same connected component of $\partial P_i\cap\partial P_0$, which is impossible by construction since to each connected component of $\partial P_i\cap\partial P_0$ is associated at most one edge.

The Euler formula on $\R ^2$ gives
$$V-E + F = 2\, ,$$
being $V$, $E$, and $F$ respectively the number of vertices, edges, and faces of the graph.
On the numbers $V, E, F$ we know the following facts: the number $V$ of vertices equals $k +1$  (because we have added the exterior polygon $P_0$ to the initial family of $k$ polygons);
moreover, the number $F$ of faces can be bounded in terms of the number $E$ of edges through the elementary inequality $3F \leq 2E$
(because, since $k \geq 2$,  every face has at least $3$ edges and every edge is on $2$ faces).
Thus, from  the Euler formula we infer
\begin{equation}\label{f11}
3k - 3 \geq E\,.
\end{equation}
The family $\mathcal E$  of all edges of the graph is the union of two disjoint subfamilies:
the subfamily  $\mathcal E _{\rm in}$ of the edges
which connect a pair of polygons $P _i$ and $P _j$ in $\{ P _1, \dots, P _k \}$, and
the subfamily   $\mathcal E _{\rm out}$
of the edges which connects $P _0$ with a polygon in $\{ P _1, \dots, P _k \}$.
Thus,  denoting respectively by  $E _{\rm in}$ and $E _{\rm out}$  the cardinalities of
such subfamilies, we have
\begin{equation}\label{f21}
E =  E _{\rm in} +  E _{\rm out}\,.
\end{equation}
We are going to estimate separately $E _{\rm in}$ and $E _{\rm out}$ in terms of the numbers
$$
\begin{array}{ll}
 & N _{\rm in}
 \text{:= total number of sides of polygons $P_i$ ($i = 1, \dots, k$),
 which intersect }
 \\ \noalign{\smallskip}
 & \hskip 1 cm \text{ the boundary of another polygon $P_j$ ($j= 1, \dots, k$) in their relative interior}
\\ \noalign{\medskip}
& N _{\rm out}
\text{:= the remaining sides of $P _1, \dots, P _k$.}
\end{array}
$$
We first consider the edges in the subfamily $\mathcal E _{\rm in}$. Let $e_{ij}$ be any such edge, connecting the vertices which correspond to $P _i$ and $P _j$.
If the intersection $\partial P _i \cap \partial P _j$ occurs in the relative interior of a side of $P _i$ (resp., such a side of $P_j$),
we associate with $e _{ij}$
such a side of $P _i$ (resp.\ a side of $P_j$). In this way,
the number of sides which are associated with an edge is at most $2$.
Notice also that the same side can be associated with more than one edge, if it contains more than one intersection between different polygons.

Therefore, we have:
\begin{equation}\label{f31}
2 E _{\rm in} \geq N _{\rm in}\,.
\end{equation}
From \eqref{f11}, \eqref{f21},  and \eqref{f31}, we deduce that
$$6k - 6  \geq 2 E  = 2   \big (  E _{\rm in} +  E _{\rm out} \big ) \geq  N _{\rm in} + 2 E _{\rm out}  \,.  $$
Therefore, to achieve the proof of statement (iii), it is enough to show that
\begin{equation}\label{f41} 2 E _{\rm out}  + 6 \geq N _{\rm out}\,.
\end{equation}
We notice that,  according to Lemma \ref{FT} (ii), every side of $P _1, \dots, P _k$ which touches {\it only} $\partial \Omega _k$ (in its relative interior) is associated to some edge in $\mathcal E _{\rm out}$. Let's count the total number $N _{\rm out}$ of such sides.

Let $e_{i}$ be an edge in $\mathcal E _{\rm out}$, connecting the vertices which correspond to $P _i$ and $P _0$, through a certain connected component of $\partial P _i  \cap \partial P _0$ containing a side of $P_i$. Let $S_1,\dots,S_m$ be the (consecutive) sides of $P_i$ corresponding to the edge $e_i$.

We distinguish three cases:
\begin{itemize}
\item Suppose that the chain $S_1,\dots,S_m$ touches $\partial\Omega_k$ only in points of the side $B_1$ of the $k$-triangle $\Omega_k$ (see Figure \ref{fig-tiling} (c)). Then, by the convexity of $P_i$, the number of sides associated to the edge $e_i$  in $\mathcal E _{\rm out}$ is at most $2$ that is, $m=1$ or $m=2$.
\item Suppose that the chain $S_1,\dots,S_m$ touches two of the sides of the $k$-triangle, say $B_1$ and $B_2$ (see Figure \ref{fig-tiling} (b)).  Then, there can be at most two sides of $P_i$ for each side of the $k$-triangle. Thus, $m\le 4$. Moreover, we can suppose that there is at most one convex polygon $P_i$ touching both $B_1$ and $B_2$ with a connected chain of the form $S_1,\dots,S_m$. In fact, let $U_{12}$ be the connected  of $\Omega_k\setminus \overline{P_i}$ whose boundary contains the vertex $B_1\cap B_2$. Then, there are no polygons in $U_{12}$. In fact, if this is not the case, then we could translate all these polygons in the direction $\nu_3$ until one of these polygons touches $P_i$. Now, in this new (still optimal) configuration, the chain $S_1,\dots,S_m$ is disconnected. As a conclusion, there are at most three polygons touching two sides of the $k$-triangle with a connected chain of sides; each of these chains contains at most $4$ segments.
\item The last case is that the chain $S_1,\dots,S_m$ touches all the three sides of the $k$-triangle. This case is outruled by performing the same translation as above.
\end{itemize}
In conclusion, we have that each edge $e_i$ connecting $P_i$ to $P_0$ contains at most two sides of $P_i$. As an exception, there might be at most three other polygons for which one of the edges has an extra contribution of $2$ more sides, so that the extra-contribution of all polygons is at most $3 \cdot 2 = 6$, which concludes the proof of \eqref{f41}.
\qed

\bigskip

\bigskip {\it Proof of Lemma \ref{t:clusters}}. Let us prove separately statements (i) and (ii).

\medskip
{\it Proof of statement (i)}.  Let us
show that,
if $\{ E _i \}$ is any convex $k$-cluster of $\Omega _k$, it holds
\begin{equation}\label{f:thesis0}\sum _{i = 1} ^ k F ( E_i) \geq k F (H)\,.
\end{equation}

\medskip
{\it Claim: To prove the above property we may assume without loss of generality that  $\Omega_k$ is a $k$-triangle.  }

Indeed, assume that we can prove the inequality \eqref{f:thesis0} for  convex $k$-clusters contained into  a $k$-triangle (for arbitrary large $k \in \N$). Let now $\Omega _k$ be
any given $k$-hexagonal structure, and let $\{ E _i\}$
be a convex $k$-cluster
of $\Omega _k$.

We observe that
$\Omega _k$ can be embedded into a larger $k'$-hexagonal structure $\Omega _{k'}$,  $\Omega _{k'} \supset \Omega _k$, such that $\Omega _{k'}$ is a $k'$-triangle.

We consider the convex $k'$-cluster $\{ E' _i\}$ of $\Omega _{k'}$ obtained by adding to the family $\{ E _i\}$ a number of $(k'-k)$ copies of $H$, namely
$E'_i$ equals $E _i$ for $1 \leq i \leq k$, and $E' _i$ equals a copy of $H$ for $k+1 \leq i \leq k'$.

By assumption, we have
$$\sum_{i=1} ^ {k'} F ( E'_i) \geq k' F (H)\,.$$
Since
$$\sum_{i=1} ^ {k'} F ( E'_i) = \sum_{i=1} ^ {k} F ( E_i) + (k' - k) F (H) \, , $$
we infer that
$$\sum_{i=1} ^ {k} F ( E_i) \geq k F (H)\,.$$

Let us now prove that inequality \eqref{f:thesis0} for any $k$-cluster $\{ E_i \}$ contained into a $k$-triangle. Let $\{ P _i \}$ be a family of polygons associated with $\{ E _i \}$ according to Lemma \ref{FT}.

Since $\sum _{i=1} ^ k  |P _i|  \leq k$, and $F (P _i) |P _i| ^ {\alpha/2} \geq \gamma (n _i)$ for every $i= 1, \dots, k$, we have
\begin{equation}\label{inizio} k ^ {\alpha/2} \sum _{i = 1} ^ k F ( P _i) \geq \Big ( \sum _{i = 1} ^ k |P _i|  \Big ) ^ {\alpha/2 } \Big ( \sum _{i=1} ^ k \frac{\gamma ( n _i)}{|P_i| ^ {\alpha/2}}\Big )\,.
\end{equation}
By applying the H\"older inequality $\sum _i | x_i y _i|  \leq (\sum _i |x_i | ^ p ) ^ {1/p}  (\sum _i |y_i | ^ {p'} ) ^ {1/p'}$ with
$$
p = \frac{\alpha +2}{2} \, , \qquad q = \frac{\alpha + 2}{\alpha}\, , \qquad x _i :=
\frac{\gamma (n _i) ^ {2/(\alpha+2) }} {|P_i| ^ {\alpha/ (\alpha + 2)}}\, , \qquad y _i = |P _i| ^ {\alpha/ (\alpha +2)}\,, $$
 we infer from \eqref{inizio} that
$$k ^ {\alpha/2} \sum _{i = 1} ^ k F ( P _i) \geq \Big (  \sum _{i=1} ^ k \gamma (n _i) ^ {2 / (\alpha +2)} \Big ) ^ {(\alpha + 2 )/2 }\,.$$
By virtue of Lemma \ref{FT} (iii) and assumption (H3) (ii), we deduce that
$$k ^ {\alpha/2} \sum _{i = 1} ^ k F ( P _i) \geq k  ^ {(\alpha + 2 )/2 } \gamma (6)\,,$$
and, in turn, that
$$\sum _{i=1} ^ k  F ( E _i) \geq \sum _{i=1} ^ k  F ( P _i) \geq k \gamma (6) \,.$$

\bigskip
{\it Proof of statement (ii)}.
In order to prove the equality $M _ k (\Omega _k) = F (H)$
we are going to show that,
if $\{ E _i \}$ is a convex $k$-cluster of $\Omega _k$ such that
\begin{equation}\label{f:ipothesis}
\max_{i= 1, \dots, k} F ( E_i) \leq F (H)\,
\end{equation}
it holds necessarily
\begin{equation}\label{f:thesis}
\max_{i= 1, \dots, k}  F ( E_i) = F (H)\,.
\end{equation}

\medskip
{\it Claim: To prove the above implication we may assume without loss of generality that  $\Omega_k$ is a $k$-triangle.  }

Indeed, assume that we can prove the implication \eqref{f:ipothesis} $\Rightarrow$ \eqref{f:thesis} for $k$-triangle. Let now $\Omega _k$ be
any given $k$-hexagonal structure, let $\{ E _i\}$
be a convex $k$-cluster
of $\Omega _k$ satisfying \eqref{f:ipothesis}, and let us show that it satisfies \eqref{f:thesis}.

We observe that, for some $k'\ge k$,
$\Omega _k$ can be embedded into a larger $k'$-triangle $\Omega _{k'}$, that is  $\Omega _{k'} \supset \Omega _k$.
We consider the convex $k'$-cluster $\{ E' _i\}$ of $\Omega _{k'}$ obtained by adding to the family $\{ E _i\}$ a number of $(k'-k)$ copies of $H$, namely
$E'_i$ equals $E _i$ for $1 \leq i \leq k$, and $E' _i$ equals a copy of $H$ for $k+1 \leq i \leq k'$.

By construction, the convex $k'$-cluster $\{ E'_i\}$ still satisfies the condition
$$\max_{i=1, \dots, k'} F ( E'_i) \leq F (H)\,.$$
Since we are assuming that the implication \eqref{f:ipothesis} $\Rightarrow$ \eqref{f:thesis} holds true for $\Omega_{k'}$, we have
$$\max_{i=1, \dots, k'} F ( E'_i) \leq F (H)\,,$$
which in turn implies  \eqref{f:thesis}.

Let us now prove the implication \eqref{f:ipothesis} $\Rightarrow$ \eqref{f:thesis} for $k$-triangles.
Let $\Omega _k$ be a $k$-triangle, and let $\{ E _i \}$ be a convex $k$-cluster of $\Omega _k$ satisfying \eqref{f:ipothesis}. Let $\{ P _i \}$ be a family of polygons associated with $\{ E _i \}$ according to Lemma \ref{FT}.

By using \eqref{f:ipothesis},  the monotonicity hypothesis (H1) made on $F$, and assumption (H3) (i),  we see that
\begin{equation}\label{f:upperP}
F ( P_i) \leq F ( E _i) \leq F (H) \qquad \forall i = 1, \dots, k\,;
\end{equation}
so we are done if we show that
\begin{equation}\label{f:conclusion}
F (P _i) \geq F (H)\qquad \forall i = 1, \dots, k\,.
\end{equation}

Let us denote by $n _i$ the number of sides of the convex polygon $P _i$.

By using \eqref{f:upperP}, assumption (H3) (i), and the definition of $\gamma (n _i)$, we have
$$\gamma (6) |P _ i | ^ {\alpha/2} \geq F ( P_i) |P _ i | ^ {\alpha/2} \geq \gamma (n _i)\qquad \forall i = 1, \dots, k\,.$$
We deduce that
\begin{equation}\label{e1}
F ( P _i) \geq \frac{\gamma (n _i) }{|P _i| ^ {\alpha/2}  }\qquad \forall i = 1, \dots, k\,
\end{equation}
and
\begin{equation}\label{e2}
| P _i| \geq \Big ( \frac{\gamma (n _i)}{\gamma (6)}  \Big ) ^ {2/\alpha}\qquad \forall i = 1, \dots, k\,.
\end{equation}
Using \eqref{e2}, summed over $k$, Lemma \ref{FT} (iii), and assumption (H3) (ii), we get
\begin{equation}\label{eq:catena}
k \geq \sum _{i = 1} ^ k |P _i| \geq \sum _{i=1} ^ k \Big ( \frac{\gamma (n _i)}{\gamma (6)}  \Big ) ^ {2/\alpha} \geq k \,.
\end{equation}
We deduce that all the inequalities \eqref{e2} hold as equalities, and then from \eqref{e1} we obtain
\begin{equation}\label{f:conc}
F ( P _i) \geq \gamma (6)\, \qquad \forall i = 1, \dots, k \,. ;
\end{equation}
finally, in view of assumption (H3) (i), we obtain the required inequalities  \eqref{f:conclusion}.
\begin{remark}\label{rem:quantitativeFK}
Note that, in case $F$ satisfies a quantitative Faber-Krahn inequality of the kind
\[
F ( P /|P|) \geq \gamma (n) + \delta_n(P/|P|, P^*_n) \quad\text{for every $n$-gon }P,
\]
where the non-negative function $\delta_n$ vanishes if and only if $P$ is regular, then one can modify
the above proofs in order to show that not only $m _ k (\Omega _k) = k F (H)$ (resp., $M _k (\Omega _k) = F (H)$), but  also that actually the unique optimal
partition consists of regular hexagons.
\end{remark}

\section{Proof of Theorem \ref{t:honeycomb2}}\label{proof2}

\begin{remark}\label{r:proof}  { Before entering into the proof  let us stress that,  in the setting of Theorem \ref{t:honeycomb2}, the assertion
$m _k (\Omega _k ) = F (H)$ for every $k$-hexagonal structure $\Omega _k$ seems to be too strong. In particular, if one tries to prove it
by adapting the arguments used in the proof of Lemma \ref{t:clusters}, one gets troubles in order to pass from the case of a triangular hexagonal structure to the case of a generic hexagonal structure. Indeed, embedding a given $k$-hexagonal structure $\Omega _k$ into a $k'$-triangle  (with $k' >k$), it is not true that for every convex $k$-partition $\{ P _i\}$ of $\Omega_k$ it is possible to find a convex $k'$-partition $\{ P' _i\}$ of $\Omega _{k'}$ such that $P '_i \cap \Omega_{k'}$ equals $P _i \cap \Omega_{k}$ for $1 \leq i\leq  k$ and $P' _i \cap \Omega_{k'}$ equals a copy of $H$ for $k+1 \leq i \leq k'$.  }
\end{remark}

The proof of Theorem \ref{t:honeycomb2} needs some preliminary lemmas.

\begin{lemma}\label{l:Euler} Let $Q$ be a convex polygon with $n _Q$ sides; let $\{ P _i \}_{\{i = 1, \dots, k \}}$ be a convex $k$-partition of $Q$, and let $n _i$ denote the number  of sides of $P _i \cap Q$. Then
$$\frac{1}{k} \sum _{i = 1} ^ k n _i \leq 6 + \frac {n _Q - 6 }{k}\,.$$
\end{lemma}

\proof Throughout the proof, we denote for brevity by $P _i$ the polygon $P _i \cap Q$, for $i = 1, \dots, k$.
Moreover,  we denote by $P _0$ the complement of $Q$ in $\R ^2$.

Let us assume without loss of generality that $k \geq 2$ (otherwise the statement is immediately satisfied).

In order to compute  the sum $\sum _{i = 1} ^ k n _i $, we construct a suitable graph associated with the family of polygons $\{P_0, \dots, P_k\}$.

More precisely, for every $i = 0, 1, \dots, k$, we associate with the polygon $P _i$ a vertex.

We denote by $\mathcal I$ the family of indices $i \in \{1, \dots , k \}$ such that $\partial P _i \cap \partial Q$ has a strictly positive $\mathcal H ^ 1$ measure.  For $i \in \mathcal I$,
some polygon $P _i$ may disconnect $Q$ (meaning that the complement in $Q$ of the interior of $P _i$  may be disconnected). Then,
for $i \in \mathcal I$, we denote by $m (i)$ the number of connected components of $\partial P _i \cap \partial Q$ having strictly positive $\mathcal H ^1$-measure.

 Then we give the following rule:
 \begin{itemize}
 \item{} for $j$ and $i$ in $\{1, \dots, k\}$, we connect the vertices corresponding to
 $P _j$ and $P _i$ if and only if their common boundary $\partial P _i \cap \partial P _j$ has a strictly positive $\mathcal H ^ 1$ measure;
 \item{} for $j=0$ and $i$ in $\{1, \dots, k\}$, we connect the vertices corresponding to $P _0$ and $P _i$ if and only if $\partial P _i \cap \partial Q$ has a strictly positive $\mathcal H ^ 1$ measure, and in this case we connect such vertices exactly with $m (i)$ edges (one edge for each connected component of $\partial P _i \cap \partial Q$ having strictly positive $\mathcal H ^1$ measure).
 \end{itemize}

For the planar graph thus constructed, the Euler formula on $\R ^2$ gives
$$V-E + F = 2\, ,$$ being $V$, $E$, and $F$ respectively the number of vertices, edges, and faces of the graph.

On the numbers $V, E, F$ we know the following facts: the number $V$ of vertices equals $k +1$  (because we have added the exterior polygon $P_0$ to the initial family of $k$ polygons);
moreover, reasoning as in the proof of Lemma \ref{FT}, the number $F$ of faces can be bounded in terms of the number $E$ of edges through the elementary inequality $3F \leq 2E$
(since we can construct a planar graph with the same number of faces and edges, which does not contain multiple edges).
Thus, from  the Euler formula we infer
\begin{equation}\label{f1}
3k - 3 \geq E\,.
\end{equation}
The family $\mathcal E$  of all edges of the graph is the union of two disjoint subfamilies:
the subfamily  $\mathcal E _{\rm in}$ of the edges
connecting a pair of polygons $P _i$ and $P _j$ in $\{ P _1, \dots, P _k \}$, and
the subfamily   $\mathcal E _{\rm out}$
of the edges connecting $P _0$ with a polygon in $\{ P _1, \dots, P _k \}$.
Thus,  denoting respectively by  $E _{\rm in}$ and $E _{\rm out}$  the cardinalities of
such subfamilies, we have
\begin{equation}\label{f2}
E =  E _{\rm in} +  E _{\rm out}\,.
\end{equation}
We are going to estimate separately $E _{\rm in}$ and $E _{\rm out}$ in terms of the numbers
$$
\begin{array}{ll}
 & N _{\rm in}
 \text{:= total number of sides of $P _1, \dots, P _k$ which do not lie on $\partial Q$ }
\\ \noalign{\medskip}
& N _{\rm out}
\text{:= total number of sides of $P _1, \dots, P _k$ which lie on $\partial Q$.}
\end{array}
$$
Since each edge of the graph  in the subfamily $\mathcal E _{\rm in}$ appears at least twice
when counting the number of sides of $P _1, \dots, P _k$ which do not lie on $\partial Q$, we have
\begin{equation}\label{f3}
2 E _{\rm in} \geq N _{\rm in}\,.
\end{equation}

On the other hand, by construction, we have
\begin{equation}\label{f4}
E _{\rm out} = \sum_ {i \in \mathcal I} m (i)\,;
\end{equation}
moreover, in the counting of $N _{\rm out}$, each connected component of $\partial P _i \cap \partial Q$ (for $i \in \mathcal I$) of strictly positive $\mathcal H ^ 1$ measure gives a contribution of at most $2$ extra sides other than the initial sides of $Q$, so that
\begin{equation}\label{f5}
N _{\rm out} - n _ Q \leq 2 \sum _{i \in \mathcal I} m (i).
\end{equation}
By \eqref{f1}, \eqref{f2}, \eqref{f3}, \eqref{f4}, and \eqref{f5}, we conclude that
$$6k - 6 + n _Q \geq 2 E+ n _Q  = 2   \big (  E _{\rm in} +  E _{\rm out} \big )+ n _Q  \geq  N _{\rm in} + 2 \sum _{i \in \mathcal I} m (i) + n _Q \geq N _{\rm in} +N _{\rm out} = \sum _{i = 1} ^ k n _i \,.  $$

\qed

\begin{lemma}\label{t:convexpol} {Under the assumptions of Theorem \ref{t:honeycomb2},  let $Q_k$ be a family of convex polygons with at most $C \sqrt k$ sides, where $C$ is a positive constant (independent of $k$). Then there exists $\overline k$ (depending on $C$ and on the constants $k _0$ and $\delta _0$ appearing in assumption (H2) (ii)) such that
$$   m _k (Q _k) \geq \frac{|Q_k|^ {\alpha /2} }{k ^ {\alpha /2}}  \gamma \Big ( 6 + \frac{ C}{\sqrt k} \Big )  \qquad \forall k \geq \overline k\,.$$
}
\end{lemma}

\proof
Let $\rho _k= \frac{k^ {1/2}}{|Q_k| ^ {1/2}}$, so that $|\rho_k Q_k | = k$.
By homogeneity of $m _k (\cdot)$, we have
$$m _k  (Q_k) = \rho_k ^ {- \alpha} m _k (\rho_k Q_k)  = \frac{k^ {- \alpha/2}}{|Q_k| ^ {- \alpha/2}}  m _k (\rho _kQ_k)\,.$$
Thus we are reduced to prove that  there exists $\overline k$ such that
$$m _k (\rho_k Q_k) \geq  \gamma \Big (  6 + \frac{ C}{\sqrt k} \Big )  \qquad \forall k \geq \overline k\,.$$

In the remaining of the proof we assume with no loss of generality that $\rho _k= 1$ and $|Q_k| = k$.

Let us show that there exists $\overline k \in \N$ such that,  if $k \geq \overline k$ and $\{ P _i \}$ is a convex $k$-partition of $Q_k$ satisfying
\begin{equation}\label{f:ipothesis2}
\max _{{i = 1, \dots, k }} F (P _i \cap Q_k) \leq \gamma \Big ( 6 +  \frac{ C}{\sqrt k}  \Big )\, ,
\end{equation}
it holds necessarily
\begin{equation}\label{f:thesis2}
\max _{{i = 1, \dots, k }} F (P _i \cap Q_k) = \gamma \Big ( 6 +  \frac{ C}{\sqrt k}  \Big )\,.
\end{equation}

Let us denote by $n _i ^k$ the number of sides of the polygon $P _i \cap  Q_k$.

We observe that, by Lemma \ref{l:Euler}, and since the number of sides $n _{Q_k}$ of $Q _k$ is at most $C \sqrt k$, it holds
\begin{equation}\label{f:meanb}
\frac{1}{k} \sum _{i = 1} ^ k  n _ i ^k \leq 6 + \frac{ n _{Q_k}-6}{k} \leq 6 + \frac{ C\sqrt k -6}{k} \leq 6 + \frac{C}{\sqrt k }\,.
\end{equation}

By using \eqref{f:ipothesis2} and the definition of $\gamma (n _i ^k)$, we have
$$ \gamma \Big ( 6 +  \frac{ C}{\sqrt k}  \Big )
  |P _ i \cap Q _k| ^ {-\alpha/2} \geq F ( P_i \cap Q_k) |P _ i  \cap Q_k| ^ {-\alpha/2} \geq \gamma (n _i ^k)\qquad \forall i = 1, \dots, k\,.$$
We deduce that
\begin{equation}\label{e12}
F ( P _i \cap Q_k) \geq {\gamma (n _i ^k) }{|P _i \cap Q_k| ^ {\alpha/2}  }\qquad \forall i = 1, \dots, k\,
\end{equation}
and
\begin{equation}\label{e22}
| P _i \cap Q_k| \leq \frac{\gamma (n _i^k)^ {-2/\alpha}}{\gamma \Big ( 6 +  \frac{ C}{\sqrt k}  \Big )^ {-2/\alpha} }  \qquad \forall i = 1, \dots, k\,.
\end{equation}
Now we sum the inequalities \eqref{e22} over $i = 1, \dots, k$.
We get
$$k = \sum _{i = 1} ^ k |P _i \cap Q_k| \leq \sum _{i=1} ^ k \frac{\gamma (n _i ^k) ^ {-2/\alpha}}{\gamma \Big ( 6 +  \frac{ C}{\sqrt k} \Big ) ^ {-2/\alpha} }   \leq k \,.$$
where the last inequality holds true  thanks to \eqref{f:meanb}
and assumption (H2) (ii), provided  $k$ is sufficiently large. More precisely it must be
$k \geq k _0$ and $C/\sqrt k  \leq \delta _0$, with $k _0$ and $\delta _0$ given by assumption (H2) (ii),  so it is enough to  take $k \geq \overline k := \max \{ k _0 , C^2 / \delta _0 ^ 2\}$.

We deduce that, for $k \geq \overline k$, all the inequalities \eqref{e22} hold as equalities, and then from \eqref{e12} we obtain
\begin{equation}\label{f:conc2}
F ( P _i \cap Q_k) \geq \gamma \Big ( 6 + \frac{ C}{\sqrt k} \Big ) \, \qquad \forall i = 1, \dots, k \,,
\end{equation}
which achieves the proof of \eqref{f:thesis2}.

 \qed
 \bigskip

To state next lemma, we need to introduce a definition. Given an open bounded and convex domain $\Omega$, and
a tiling $(H_i ) _{i \in I}$  of $\R^2$ made by copies $H _i$ of the unit area regular hexagon $H$,
as done in the proof of Theorem \ref{t:honeycomb}, we set
$$
\begin{array}{ll}
& I ^ {\inte} (\rho, \Omega)  := \Big \{ i \in I \ :\ H _ i \subset (\rho \Omega) \Big \} \,,
\\  \noalign{\bigskip}
& \rho ^ {\inte} (k, \Omega)  := \inf \Big \{ \rho >0 \ : \ \sharp I ^ {\inte} (\rho, \Omega) \geq k \Big \} \end{array}
$$
(for any positive $\rho$,   $\rho \Omega$ denotes  the dilation of $\Omega$ of factor $\rho$).

Since $\Omega$ is fixed, hereafter
we shorten the notation into $\rho _k = \rho ^ {\inte} (k, \Omega)$, and $ I ^ {\inte} (\rho )=  I ^ {\inte} (\rho, \Omega) $.

We observe that, setting
$$I ^ {\inte}_{\partial \Omega} (\rho )  := \Big \{ i \in I ^ {\inte} (\rho )  \ :\ H _ i \cap \partial (\rho \Omega) \neq \emptyset \Big \}\,,$$
it holds
\begin{equation}\label{difference}\sharp I ^ {\inte} (\rho_k ) - \sharp   I ^ {\inte}  _{\partial \Omega} (\rho_k )  < k\,.
\end{equation}
Indeed, otherwise $\rho _k \Omega$ would contain at least $k$ hexagons (not touching $\partial (\rho _k \Omega)$), contradicting the minimality of $\rho _k$  among the radii $\rho$ such that $\sharp I ^ {\inte} (\rho) \geq k$.

Then we can remove from the family of all the hexagons covered by $\rho _k \Omega$ some ones, all of them touching $\partial(\rho _k \Omega)$, in such a way that the remaining number is exactly $k$.  Notice that the choice of the hexagons touching $\partial(\rho _k \Omega)$ which can be removed is possibly not unique, but thanks to  \eqref{difference} there is at least one.

Keeping the above notation, we give the following

\begin{definition}\label{inner-tiling} We call an {\it inner $k$-hexagonal structure of}  $\Omega$ any $k$-hexagonal structure contained into $\rho _k \Omega$
obtained
as described above, {\it i.e.},
by
removing some hexagons touching $\partial (\rho _k \Omega)$ from the family of all hexagons contained into $\rho _k \Omega$.
\end{definition}

\begin{lemma}\label{convex-envelope}  Let  $H _k (\Omega)$ be an inner $k$-hexagonal structure of $\Omega$, and let ${\rm  conv} (H _k (\Omega))$ denote its convex envelope.  Then
 \begin{itemize}
 \item[(i)]
the number of  sides of ${\rm  conv} (H _k (\Omega))$  does not exceed $C \sqrt k$,  being $C$ a positive constant depending only on $\Omega$;

\smallskip
\item[(ii)] it holds
$$\lim _{k \to + \infty} \frac{\big |{\rm  conv} (H _k (\Omega) ) \big |  }{| \rho _k \Omega|} = 1 
$$
\end{itemize}
\end{lemma}

\proof (i) We denote by $p$ the number of hexagons in $H _k(\Omega)$ having a free side, meaning a side which is not in common with another hexagon in
$H _k(\Omega)$. Clearly the number of sides of ${\rm  conv} (H _k (\Omega))$ is not larger than $6p$. So  we are going to estimate $p$.
We observe that, if $h$ is a copy of $H$ lying in $H _k(\Omega)$  and having a free edge, it holds necessarily
\begin{equation}\label{inclusion}
h \subseteq \partial (\rho _k \Omega) \oplus B _4  : = \big \{ x \in \R ^ 2 \ :\ {\rm dist} ( x, \partial (\rho _k \Omega)) \leq 4 \big \} \,.
\end{equation}
Indeed, if  $h'$ is a copy of $H$ lying outside $H _k(\Omega)$  and having a side in common with $h$,  by construction  $h'$ cannot be entirely contained into $\rho _k \Omega$ (because otherwise $h'$ will be an hexagon in $H _k (\Omega)$).  Therefore any such hexagon $h'$ meets necessarily the boundary of $\rho _k \Omega$, and the inclusion \eqref{inclusion} follows, since the diameter of $H$ is less than $2$.

In view of \eqref{inclusion}, and since any hexagon with a free edge has unit area, we have
$$p \leq |  \partial (\rho _k \Omega) \oplus B _4  |\,.$$
Now we observe that
$$  | \partial (\rho _k \Omega) \oplus B _4 |  = \rho _k ^ 2 |  \partial  \Omega \oplus B _{4/\rho _k} | \leq  \rho _k ^ 2 M  \frac{4}{\rho _k} \mathcal H ^ 1 (\partial \Omega)   = 4 M \mathcal H ^ 1 (\partial (\rho _k \Omega))\,,$$
where the last inequality holds for some positive constant $M= M (\Omega)$ independent of $k$. Indeed, by the Lipschitz regularity assumed on $\partial \Omega$, the  perimeter $\mathcal H ^ 1 (\partial \Omega)$  agrees with the so-called Minkowski content of $\partial \Omega$, namely with $\lim _{\varepsilon \to 0}  (2\varepsilon )^ {-1} |  \partial  \Omega \oplus B _{\varepsilon} | $.

We deduce that
$$ p \leq 4 M \mathcal H ^ 1 (\partial \Omega) \rho _k \leq C \sqrt k\,,$$
where the last equality holds for some positive constant $C = C (\Omega)$ since  $\rho _k \sim \sqrt k$ as $k \to + \infty$ ({\it cf.}  the proof of Theorem \ref{t:honeycomb}).

\medskip

(ii)  Let us denote for brevity $Q _k  :={\rm conv} ( H _k (\Omega))$.  Since
$Q _k \subseteq \rho _k \Omega$, we have immediately
$$\limsup _{ k \to + \infty } \frac{ |Q_k| } {|\rho _k \Omega|} \leq 1\,.$$
We have to prove that also the converse estimate holds true.
Since  \eqref{inclusion} is satisfied for every $h \in H _k (\Omega)$ having a free edge, it holds
$$\rho _k \Omega \subseteq H _k (\Omega) \cup \big [\partial (\rho _k \Omega) \oplus B _4 \big ]\,,$$
so that
$$|\rho _k \Omega | \leq |  H _k (\Omega)  | +  | \partial (\rho _k \Omega) \oplus B _4 | \leq |Q_k|+  | \partial (\rho _k \Omega) \oplus B _4 | \,.$$
We have already shown in part (i) of the proof that
$$  | \partial (\rho _k \Omega) \oplus B _4 | \leq  4 M \mathcal H ^ 1 (\partial (\rho _k \Omega))\,,$$
for some positive constant $M = M (\Omega)$ independent of $k$.
Therefore we have
$$\begin{array}{ll}
\displaystyle \liminf _{ k \to + \infty } \frac{ |Q_k| } {|\rho _k \Omega|}   \geq &  \displaystyle \liminf _{ k \to + \infty } \frac{|\rho _k \Omega |- | \partial (\rho _k \Omega) \oplus B _4 |} {|\rho _k \Omega|}
\\ \noalign{\medskip}
 &  \displaystyle \geq  \liminf _{ k \to + \infty } \Big ( 1-  \frac{4 M \mathcal H ^ 1 (\partial (\rho _k \Omega))} {|\rho _k \Omega|}  \Big )
\\ \noalign{\medskip}
 &  \displaystyle =  \liminf _{ k \to + \infty } \Big ( 1-  \frac{4 M  \rho _k \mathcal H ^ 1 (\partial  \Omega)} { \rho _k ^ 2 |\Omega|}  \Big )  = 1\,,
\end{array}
 $$
where the last equality holds  since  $\rho _k \sim \sqrt k$ as $k \to + \infty$.

\qed

\bigskip\bigskip {\it Proof of Theorem \ref{t:honeycomb2}}.

Let $(H_i ) _{i \in I}$ denote a tiling of $\R^2$ made by copies $H _i$ of the unit area regular hexagon $H$.

In the same way as in the proof of Lemma \ref{t:helffer},
for any positive factor of dilation $\rho$,  we introduce the families of indices
$I ^ {\inte} (\rho, \Omega)$ and
$I ^ {\ext} (\rho, \Omega)$; moreover,
for every $k \in \N$, we define the radii $\rho ^ {\inte} (k, \Omega) $ and $\rho ^ {\ext} (k, \Omega)$, and we recall that they behave asymptotically as $\frac{\sqrt k}{\sqrt {|\Omega|}} $ as $k \to + \infty$ ({\it cf}. \eqref{f:ae}).

We observe that, since $F$ satisfies assumption (H2), the map $\Omega \mapsto m _k (\Omega)$ is homogeneous of degree $\alpha$ under dilations. Moreover,  if
$\Omega \subseteq \Omega'$, it holds $ \P _k (\Omega') \subseteq \P _k (\Omega )$ and, since
$F$ satisfies assumption (H1), we  have
$$\Omega \subseteq \Omega' \quad \Rightarrow \quad F (P _i \cap \Omega) \leq F ( P _i \cap \Omega ' ) \qquad \forall \, \{ P _i \} \in \P _k (\Omega') \,.$$
Consequently, the map $\Omega \mapsto m _k (\Omega)$ is monotone increasing under inclusions.

We are ready to give an upper and lower bound for $m _k (\Omega)$.

\medskip
{\it Upper bound}.
We take $\rho = \rho ^ {\ext} (k, \Omega)$.

By using  the homogeneity and increasing monotonicity of $m _k (\cdot)$,
we get
$$m_k (\Omega) = \rho^ {-\alpha} m _k ( \rho \Omega) \leq \rho^ {-\alpha} m _k (\Omega_k)\,,$$
where $\Omega _k$ denotes a $k$-hexagonal structure;
moreover, since there exists a convex $k$-partition $\{ P _i \} \in \P_k (\Omega_k)$ having among its elements a convex polygon whose intersection with $\Omega _k$ is a copy of $H$,  we have
$m _k (\Omega_k ) \leq F (H)$.
We infer that
\begin{equation}\label{f:ub2} \limsup _{k \to + \infty}   {k ^ {\alpha/2}} {m_k (\Omega)} \leq \limsup _{k \to + \infty}  {k ^ {\alpha/2}}  {\big( \rho^ {\ext}(k, \Omega)  \big ) ^ {-\alpha} }  F (H) = |\Omega| ^ {{\alpha}/{2}}   F (H)  \,,
\end{equation}
where in the last equality we have exploited  the fact that $ \rho^ {\ext}(k, \Omega)  \sim \frac{\sqrt k}{\sqrt {|\Omega|}}$ as $k \to + \infty$.

\medskip
{\it Lower bound}.
We take $\rho = \rho ^ {\inte} (k, \Omega)$, and we choose an inner $k$-hexagonal structure $H _k (\Omega)$ according to Definition \ref{inner-tiling}.
We set $Q_k := {\rm conv} ( H _k (\Omega))$.  Since $\rho \Omega \supseteq Q _k$ (because $\rho \Omega \supseteq H _k (\Omega)$ and we assumed $\Omega$ convex),
by homogeneity and  increasing monotonicity of $m _k (\cdot)$, we have
$$m _k ( \Omega) = \rho ^{- \alpha} m _k (\rho \Omega) \geq \rho ^ {-\alpha}  m _{k }  (Q _{k } )
\,.$$
By Lemma \ref{convex-envelope} (i), the number of sides of $Q_k$ is not larger than $C \sqrt k$, for a positive constant $C = C (\Omega)$.
 Then, by Lemma \ref{t:convexpol},  there exists $\overline k$ such that
$$  m _k (Q _k) \geq \frac {|Q_k|^ {\alpha /2} } {k ^ {\alpha /2}}   \gamma \Big ( 6 + \frac{ C}{\sqrt k} \Big )  \qquad \forall k \geq \overline k\,.$$
Thus we have
$$   {k ^ {\alpha/2}} {m_k (\Omega)} \geq
  {k ^ {\alpha/2}}  {\big( \rho^ {\inte}(k, \Omega)  \big ) ^ {-\alpha} }  m _{k }  (Q _{k } )
 \geq  {\big( \rho^ {\inte}(k, \Omega)  \big ) ^ {-\alpha} }  {|Q_k|^ {\alpha /2} }   \gamma \Big ( 6 + \frac{ C}{\sqrt k} \Big )
$$
Now, we pass to the liminf as $k \to + \infty$ in the above inequality. By applying Lemma \ref{convex-envelope} (ii),  and recalling the assumption that $\gamma$ is continuous at $6$ with $\gamma (6) = F (H)$, we conclude that
\begin{equation}\label{f:lb}
\liminf _{k \to + \infty}   {k ^ {\alpha/2}} {m_k (\Omega)} \geq   |\Omega| ^ {\alpha/2} \liminf _{k \to + \infty}   \gamma \Big ( 6 + \frac{ C}{\sqrt k} \Big )    = |\Omega| ^ {\alpha/2} F (H)\,.
\end{equation}

\medskip
The proof is archived by combining \eqref{f:ub2} and \eqref{f:lb}. \qed

\section {Proof of Proposition \ref{p:l1}}\label{proof3}

The functional $F (\Omega) = \lambda _ 1 (\Omega)$ satisfies the hypotheses (H1) and (H2) of both Theorems  \ref{t:honeycomb0} and \ref{t:honeycomb}. Assumption \eqref{i1} corresponds to the hypothesis (H3) (i) of those theorems.
Next lemma shows that assumptions \eqref{i1}-\eqref{i2} ensure the validity of hypothesis (H3) (ii). Consequently, Proposition \ref{p:l1} follows.

\begin{lemma}\label{f:implication3}
Under the assumptions \eqref{i1}-\eqref{i2}, the map $n \mapsto \gamma (n)$ defined by \eqref{f:gl1} satisfies
\begin{equation}\label{f:implication}
\displaystyle \frac{1}{k} \sum _{i=1} ^ k n _i \leq 6 \ \Rightarrow \ \frac{1}{k}  \sum _{i=1} ^ k \gamma^\frac 12 (n _i )  \geq  \gamma^\frac 12 (6 ) \,.
\end{equation}
\end{lemma}
Before proving this lemma, let us observe that if \eqref{f:implication} holds, then one also gets
\begin{equation}\label{f:implication4}
\displaystyle \frac{1}{k} \sum _{i=1} ^ k n _i \leq 6 \ \Rightarrow \ \frac{1}{k}  \sum _{i=1} ^ k \gamma (n _i )  \geq  \gamma (6 ) \,
\end{equation}
as a direct consequence of Cauchy-Schwartz inequality. It is important to notice that \eqref{f:implication4} may be proved to hold even in the absence of \eqref{f:implication}. This could occur in the case in which the estimates on the values $a,b$ in \eqref{i2} are not fine enough for \eqref{f:implication}, but good for  \eqref{f:implication4}. The proof of \eqref{f:implication4} follows step by step Lemma \ref{f:implication3}.

The assumptions  $a= 6.022 \pi$ and $b = 5.82\pi$ in \eqref{i2} are good enough to prove Lemma \ref{f:implication3}, but there is some flexibility. The precise requirements for $a$ and $b$ will naturally follow from the proof.

We refer the reader to \cite{rsj16} for a precise computation of the eigenvalues.
\proof
We prove the statement by induction over $k$.

Assume $k= 1$. We have to show that $n \leq 6$ $\Rightarrow$ $\gamma (n)  \geq \gamma (6)$.
This is straightforward, since for every $n \leq 6$ an optimal polygon $P _n ^ {\rm opt}$ for the minimization problem which defines $\gamma (n)$  can be approximated in Hausdorff distance by a sequence $\{H _j\}$ of convex hexagons, so that
$$\gamma (n) = \lambda _ 1  (P _n ^ {\rm opt} ) | P _n ^ {\rm opt}|  = \lim _j \lambda _ 1  (H _j ) | H _j| \geq \gamma (6)\,.$$

We now assume that \eqref{f:implication} holds true for a certain $k \in \N$, and we deduce it for $k+1$.

Given $n _1, \dots, n _{k+1} $ satisfying
\begin{equation}\label{f:meanD}
\frac{1}{k+1} \sum _{i=1} ^ {k+1} n _i \leq 6 \, ,
\end{equation}
let us show that
\begin{equation}\label{f:th}
\frac{1}{k+1}  \sum _{i=1} ^ {k+1} \gamma^\frac12 (n _i )  \geq  \gamma^\frac12 (6 )\,.
\end{equation}
Without loss of generality, we assume that
$$\max _{i = 1, \dots , k+1} n _i = n _1 \qquad \text{ and } \qquad \min _{i = 1, \dots , k+1} n _i = n _2\,.$$
The idea is to use an exact estimate for a small number of integers, including $n _1$ and $n _2$, which have average at least $6$, and to use the induction argument for the remaining integers.
For convenience, let us list below the inequalities we are going to exploit:
\begin{eqnarray}
\gamma ^\frac12(3) = \lambda _1 ^\frac12(P_3 ^*) = \Big (\frac{4 \pi ^ 2}{\sqrt 3} \Big ) ^\frac12\geq 2.693 \pi ^\frac12
& \label{g3} \\
\noalign{\medskip}
\gamma^\frac12 (4) = \lambda _ 1^\frac12 ( P _ 4 ^*) = \sqrt{2} \pi  \geq 2.506 \pi ^\frac12
& \label{g4} \\
\noalign{\medskip}
\gamma ^\frac12 (5) \geq a ^\frac12 \ge  2.4539\pi ^\frac12  \qquad \qquad \quad \
& \label{g5} \\
\noalign{\medskip}
\gamma ^\frac12 (6)  = \lambda _ 1 ^\frac12 ( P_ 6 ^* ) \leq 2.433 \pi ^\frac12 \qquad \quad
& \label{g6}
\\
\noalign{\medskip}
\gamma ^\frac12 (7) \geq b ^\frac12 \ge 2.4124 \pi ^\frac12
\,. \qquad \qquad \quad  \ \
& \label{g7}
\\
\noalign{\medskip}
\lambda_1 ^\frac12 (B) \geq 2,404 \pi ^\frac12
\,. \qquad \qquad \quad  \ \
& \label{gb}
\end{eqnarray}
Above, $B$ is the ball of unit area whose value is explicitly known in term of the Bessel function $J_0$ and is larger than $5.783 \pi$.
Notice that \eqref{g3}-\eqref{g4} hold true since the regular triangle $P_3^*$ and the square
$P_4^*$ minimize $\lambda _1$ among triangles and quadrilaterals of given area (the proof by Steiner symmetrization can be found for instance in \cite[Section 3]{H06}); on the other hand, \eqref{g5}-\eqref{g7} are
 exactly our assumption \eqref{i2}. Inequality \eqref{g6} is a consequence of hypothesis \eqref{i1} associated to a numerical estimate from above of the eigenvalue on the regular hexagon.

We are going to distinguish the three cases $n _1 \geq 9$, $n _1 = 8$, and $n _ 1 = 7$.

\medskip
$\bullet$ {\it  Case 1: $n_1 \geq 9$}.

\smallskip
Clearly $n _2 <6$, and we distinguish the three subcases $n _2 = 3, 4, 5$.

\medskip
-- {\it Subcase (1a): $n _2 = 3$}.

Since $n _1 + n _ 2 \geq 12$, by \eqref{f:meanD} we have $\sum _{i = 3} ^ {k+1} n _i  \leq (k-1) 6$; hence, from the induction hypothesis,  we infer that
$ \sum _{i=3} ^ {k+1} \gamma^\frac12 (n _i )  \geq  (k-1) \gamma^\frac12 (6 )$.
The thesis follows by adding to the previous inequality the following one:
\begin{equation}\label{f:add}
\gamma^\frac12 (n _1) + \gamma^\frac12 (n _2) \geq \lambda _ 1 (B) + \gamma^\frac12 (3) \geq 2,404 \pi ^\frac12 + 2.693 \pi ^\frac12  \geq 2 \gamma^\frac12 (6)\, .
\end{equation}
Here we have used \eqref{g3} to estimate $\gamma (n _2)$, and the Faber Krahn inequality  in order to bound from below  $\gamma (n_1)$ with the first Dirichlet eigenvalue of the ball $B $ of unit area.   The last inequality in \eqref{f:add} follows from \eqref{g6}.

\medskip
-- {\it Subcase (1b): $n _2 = 4$}.

We repeat the same argument as in Subcase (1a), with the inequality \eqref{f:add} replaced by the following one, obtained from \eqref{g4}:
\begin{equation}\label{f:add2}
\gamma^\frac12 (n _1) + \gamma^\frac12 (n _2) \geq \lambda _ 1^\frac12 (B) + \gamma^\frac12 (4) \geq 2,404 \pi ^\frac12 + 2.506 \pi ^\frac12  \geq 2 \gamma^\frac12 (6)\,.
\end{equation}

\medskip
-- {\it Subcase (1c): $n _2 = 5$}.

By the definition of $n _2$ we have that, for every $i \geq 2$, $n_i <6$ implies $n _i = 5$.
Moreover, since the average of the $n_i$'s  does not exceed $6$, and $n _1 \geq 9$, there exist at least another integer, say $n _3$, such that $n _ 3 = 5$.

Since $n _1 + n _ 2 + n _3 \geq 9 + 5+ 5 > 18$, by \eqref{f:meanD} and the induction hypothesis we have
$
 \sum _{i=4} ^ {k+1} \gamma^\frac12 (n _i )  \geq  (k-2) \gamma^\frac12 (6 )$.

We conclude by adding  to the previous inequality the following one:
\begin{gather}\label{f:add1}
\gamma^\frac12 (n _1) + \gamma^\frac12 (n _2)  + \gamma ^\frac12(n _3) \geq \lambda ^\frac12_ 1 (B) + 2 \gamma^\frac12 (5) \geq 2,404 \pi ^\frac12 + 2a^\frac12  \\
\geq 2,404 \pi ^\frac12 + 2 \cdot   2.4539\pi ^\frac12 \geq 3 \gamma ^\frac12(6) \nonumber \,.
\end{gather}
Here we have used the Faber-Krahn inequality, the assumption \eqref{i2} made on $\gamma (5)$, and \eqref{g6}.

\medskip
$\bullet$ {\it  Case 2: $n_1 = 8$}.

\medskip
-- {\it Subcase (2a): $n _2 = 4$}.

Since $n _ 1+ n _ 2 \geq 12$, we proceed as in cases  (1a) and (1b).  Indeed, the same inequality as in \eqref{f:add2} holds, and
the thesis follows as usual by addition and exploiting the induction hypothesis.

\medskip
-- {\it Subcase (2b): $n _2 = 5$}.

By arguing as done as in case (1c), we see that there exists at least another integer, say $n _3$, such that $n _ 3 = 5$.
We have $n _1 + n _2 + n _ 3 = 18$. Then we can conclude  as done in case (1c), since
$
 \sum _{i=4} ^ {k+1} \gamma^\frac12 (n _i )  \geq  (k-2) \gamma^\frac12 (6 )$, and the
the same inequality as in \eqref{f:add1} is in force.

\medskip
-- {\it Subcase (2c): $n _2 = 3$}.

We distinguish two further subcases:

(i) If in the family $\{ n _3, \dots, n _{k+1}\}$ there exists at least an integer, say $n _3$, such that $n _ 3  \in \{7, 8 \}$, then we have
$n _1 + n _2 + n _3 \geq 8 + 3 + 7 = 18$.  Then the induction hypothesis ensures that $
 \sum _{i=4} ^ {k+1} \gamma^\frac12 (n _i )  \geq  (k-2) \gamma^\frac12 (6 )$, and we conclude by adding the inequality
$$\gamma^\frac12 (n _1 ) + \gamma ^\frac12(n _2 ) + \gamma^\frac12 (n _3 ) \geq \lambda^\frac12 _ 1 (B) + \gamma^\frac12 (3) + \lambda^\frac12 _ 1 (B) \geq 3 \gamma^\frac12 (6) \,.$$

(ii) If $n _i \leq 6$ for every $i \in \{3, \dots, k +1\}$, by  \eqref{g3}-\eqref{g4}-\eqref{g5} we have $\gamma (n _i) \geq \gamma (6)$ for every $i \in \{3, \dots, k +1\}$. Then the inequality $\sum _{i = 3 } ^ { k+1}  \gamma^\frac12 (n _i) \geq (k-2) \gamma ^\frac12(6)$  holds true, and
  the thesis follows by adding the  inequality
$$\gamma^\frac12 (n _1 ) + \gamma^\frac12 ( n_2) \geq \lambda ^\frac12_ 1 (B) + \gamma^\frac12 (3) \geq  2 \gamma^\frac12 (6)\, . $$

\medskip
$\bullet$ {\it  Case 3: $n_1 = 7$}.

\medskip
-- {\it Subcase (3a): $n _2 = 5$}.

Since $n _1 + n _ 2 = 12$,  we proceed as in cases (1a), (1b), and (2a). Namely, by \eqref{f:meanD}, we have $\sum _{i = 3} ^ {k+1} n _i  \leq (k-1) 6$; hence, by induction hypothesis, we have $\sum _{i = 3 } ^ { k+1}  \gamma^\frac12 (n _i) \geq (k-2) \gamma ^\frac12(6)$. The thesis follows by adding the inequality
$$\gamma^\frac12(n _1 )  + \gamma^\frac12 (n _ 2) = \gamma^\frac12 (7) + \gamma ^\frac12(5) \geq a^\frac12 + b^\frac12 \geq   2.4539\pi ^\frac12 + 2.4124 \pi ^\frac12 \geq 2 \gamma^\frac12 (6)\, , $$
which follows from our assumptions \eqref{i1}-\eqref{i2}.

\medskip
-- {\it Subcase (3b): $n _2 = 4$}.

We distinguish two further subcases:

(i) If in the family $\{ n _3, \dots, n _{k+1}\}$ there exists at least an integer, say $n _3$, such that $n _ 3 =7$, then we have
$n _1 + n _2 + n _3 \geq 7 + 4 + 7 = 18$.  Then by induction hypothesis we have $
 \sum _{i=4} ^ {k+1} \gamma ^\frac12(n _i )  \geq  (k-2) \gamma^\frac12 (6 )$, and we conclude by adding the inequality
\begin{gather}
\gamma^\frac12 (n _1 ) + \gamma ^\frac12(n _2 ) + \gamma ^\frac12(n _3 ) = 2 \gamma^\frac12 (7) + \gamma ^\frac12(4)  \geq  2 b^\frac12 +  2.506 \pi ^\frac12 \geq \nonumber \\
2\cdot 2.4124 \pi ^\frac12 + 2.506 \pi ^\frac12 \geq 3 \gamma^\frac12 (6) \nonumber\,.
\end{gather}

(ii) If $n _i \leq 6$ for every $i \in \{3, \dots, k +1\}$, by  \eqref{g3}-\eqref{g4}-\eqref{g5} we have $\gamma (n _i) \geq \gamma (6)$ for every $i \in \{3, \dots, k +1\}$. Then  the thesis follows by adding the two inequalities $\sum _{i = 3 } ^ { k+1}  \gamma^\frac12 (n _i) \geq (k-2) \gamma^\frac12 (6)$ and
\begin{gather}
\gamma^\frac12 (n _1 ) + \gamma^\frac12 ( n_2) = \gamma ^\frac12(7) + \gamma^\frac12 (4) \geq b^\frac12 + 2.506 \pi ^\frac12 \geq \nonumber \\  2.4124 \pi ^\frac12 + 2.506 \pi ^\frac12 \geq  2 \gamma^\frac12 (6)\nonumber\, .
\end{gather}

\medskip
-- {\it Subcase (3c): $n _2 = 3$}.

We have to distinguish three subcases:

(i) Assume that  in the family  $\{ n _3, \dots, n _{k+1}\}$ there exists at least two integers, say $n _3$ and $n _4$, such that $n _ 3 = n _ 4= 7$. Then we have $n _1 + n _2 + n _3 + n _4 = 24$.  By \eqref{f:meanD}
we have $\sum _{i = 5} ^ {k+1} n _i  \leq (k-4) 6$; hence, from the induction hypothesis,  it holds
$
 \sum _{i=5} ^ {k+1} \gamma^\frac12 (n _i )  \geq  (k-4) \gamma^\frac12 (6 )$.
 We conclude by adding the estimate
\begin{gather}
\gamma ^\frac12(n _1) + \gamma ^\frac12(n _2 ) + \gamma^\frac12 (n _3)  + \gamma ^\frac12(n _4) \geq 3 \gamma ^\frac12(7) + \gamma ^\frac12(3) \geq 3 b^\frac12 + \gamma^\frac12 (3) \geq \nonumber \\ 3\cdot 2.4124 \pi ^\frac12 +  2.693 \pi ^\frac12 \geq 4 \gamma^\frac12 (6)\nonumber\,.
\end{gather}

(ii) Assume that in the family  $\{ n _3, \dots, n _{k+1}\}$ there exists one integer, say $n _3$, such that $n _ 3 = 7$, while $n _i \leq 6$ for all $i \in \{4 , \dots, k +1\}$. Since we have already seen in subcase (i) of case (3b) that $2 \gamma ^\frac12(7) + \gamma ^\frac12(4) \geq 3 \gamma ^\frac12(6)$, and since $\gamma (3) > \gamma (4)$, a fortiori we have $2 \gamma^\frac12 (7) + \gamma^\frac12 (3) \geq 3 \gamma ^\frac12(6)$.  Then  the thesis follows by adding the two inequalities $\sum _{i = 4 } ^ { k+1}  \gamma^\frac12 (n _i) \geq (k-3) \gamma ^\frac12(6)$ and
$$\gamma ^\frac12(n _1 ) + \gamma ^\frac12( n_2) + \gamma^\frac12 (n _3) = 2 \gamma ^\frac12(7) + \gamma ^\frac12(3) \geq 3 \gamma^\frac12(6)\, . $$

(ii) Eventually, assume that $n _i \leq 6$ for every $i \in \{3, \dots, k +1\}$.
Since we have already seen in subcase (ii) of case (3b) that $ \gamma^\frac12 (7) + \gamma^\frac12 (4) \geq 2 \gamma^\frac12 (6)$, and since $\gamma (3) > \gamma (4)$, a fortiori we have $\gamma^\frac12 (7) + \gamma ^\frac12(3) \geq 2 \gamma^\frac12 (6)$.
 Then  the thesis follows by adding the two inequalities $
\sum _{i = 3 } ^ { k+1}  \gamma^\frac12 (n _i) \geq (k-2) \gamma^\frac12 (6)$ and
$$\gamma ^\frac12(n _1 ) + \gamma^\frac12 ( n_2)  =  \gamma ^\frac12(7) + \gamma ^\frac12(3) \geq 2 \gamma ^\frac12(6)\,.$$

\qed
\bigskip

\section{Appendix}\label{app}

\begin{lemma}\label{giamma0}
The function
\[
\psi(t)=\left(\frac{2t\sin(\pi/t)+\sqrt{2\pi t\sin(2\pi/t)}}{\sqrt{2 t\sin(2\pi/t)}}\right)^{2/3}
\]
is decreasing and strictly convex on $[3,+\infty)$.
\end{lemma}
\begin{proof}
Writing
\[
g(s)=\frac{\tan s}{s},\qquad h(s)= 1+\sqrt{g(s)},
\]
we obtain that
\[
\psi(t) = \left(\sqrt{t\tan(\pi/t)}+\sqrt\pi\right)^{2/3}
= \pi^{1/3}\, [h(s(t))]^{2/3},
\]
where
\[
s=s(t)=\frac{\pi}{t}\in\left(0,\frac{\pi}{3}\right],\qquad s'=-\frac{\pi}{t^2}=-\frac{s^2}{\pi}<0, \qquad
s''=\frac{2\pi}{t^3}=\frac{2s^3}{\pi^2}>0.
\]
Direct calculations show that, for $s\in (0,\pi/3]$,
\[
\begin{split}
g'(s)&=\frac{1}{s\cos^2s}-\frac{\tan s}{s^2} = \frac{2s-\sin 2s}{2s^2\cos^2s}>0,\smallskip\\
g''(s)&=\frac{2\sin s}{s\cos ^3s}-\frac{2 }{s^2\cos ^2s}+ \frac{2 \tan s}{s^3} =
2\frac{s^2 \sin s - s \cos s + \sin s \cos^2 s}{s^3\cos ^3s}.
\end{split}
\]
Since
\[
h'(s)=\frac{g'(s)}{2\sqrt{g(s)}}>0,\qquad
h''(s)=\frac{g''(s)}{2\sqrt{g(s)}} - \frac{[g'(s)]^2}{4[g(s)]^{3/2}}{4[g(s)]^{3/2}},
\]
on the considered interval, on the one hand we infer
\[
\pi^{-1/3}\psi'(t)= \frac23[h(s(t))]^{-1/3}h'(s(t))\cdot s'(t)<0,
\]
so that $\psi$ is decreasing. On the other hand, we are going to show the positivity of
\[
\begin{split}
\frac{3\pi^{-1/3}}{2}\psi''(t)&= -\frac13[h(s)]^{-4/3}[h'(s)\cdot s']^2 +
[h(s)]^{-1/3}h''(s)\cdot [s']^2 + [h(s)]^{-1/3}h'(s)\cdot s'' \\
&>
\frac13 h^{-4/3}[s']^2 \left\{3h\cdot h'' - [h']^2\right\}\\
&=
\frac13 h^{-4/3}[s']^2\left\{3\left(1+\sqrt{g}\right)\cdot \frac{2g\cdot g'' - [g']^2}{4g^{3/2}}
- \frac{[g']^2}{4g}\right\}\\
&>
\frac13 h^{-4/3}[s']^2\frac{1+\sqrt{g}}{g^{3/2}}\left\{\frac32 g\cdot g'' - [g']^2\right\},
\end{split}
\]
where
\[
\begin{split}
\frac32 g\cdot g'' - [g']^2 &= \frac{3\sin s(s^2 \sin s - s \cos s + \sin s \cos^2 s) - (s  - \sin s \cos s)^2}{s^4\cos ^4s}\\
&=\frac{3s^2 \sin^2 s - s\sin s \cos s + 2 \sin^2 s \cos^2 s - s^2 }{s^4\cos ^4s}\\
&=\frac{\frac12 s^2 - \frac32 s^2 \cos 2s - \frac12 s\sin 2s + \frac14 - \frac14 \cos 4s }{s^4\cos ^4s}
=:\frac{k(s)}{s^4\cos ^4s}.
\end{split}
\]
By direct calculations we infer $k(0)=k'(0)=0$ and, for $s\in (0,\pi/4]$,
\[
\begin{split}
k''(s)
&= 1 + (-5 + 6 s^2) \cos 2 s + 14 s \sin 2 s + 4 \cos 4 s  \ge 1 -5 \cos 2 s + 7 \sin^2 2 s + 4 \cos 4 s\\
&= \frac92 -5 \cos 2 s + \frac12 \cos 4 s \ge \frac92 -5 \left(1-\frac{(2 s)^2}{2} +
\frac{(2 s)^4}{4!}\right) + \frac12 \left(1-\frac{(4 s)^2}{2}\right)\\
&= 6s^2 - \frac{10}{3}s^4 >0;
\end{split}
\]
Finally, also when $s\in [\pi/4,\pi/3]$
\[
\begin{split}
k''(s)
&  \ge
1-\frac12\left(-5 + 6 \frac{\pi^2}{9}\right) + 7\pi\sqrt3 - 4>0,
\end{split}
\]
and also the convexity of $\psi$ follows.
\end{proof}
\begin{lemma}\label{giamma1}
The function
\[
\varphi(t)=\left(\frac{2t\sin(\pi/t)+\sqrt{2\pi t\sin(2\pi/t)}}{\sqrt{2 t\sin(2\pi/t)}}\right)^2
\]
is decreasing and strictly convex on $[3,+\infty)$.
\end{lemma}
\begin{proof}
The proof is a direct consequence of Lemma \ref{giamma0}. Indeed, keeping the
corresponding notation, we have
\[
\begin{split}
\varphi(t)&= [\psi(t)]^3,\\
\varphi'(t)&= 3[\psi(t)]^2\psi'(t)<0,\\
\varphi''(t)&= 6\psi(t)[\psi'(t)]^2 + 3[\psi(t)]^2\psi''(t)>0,
\end{split}
\]
as long as $t\ge 3$.
\end{proof}

\begin{lemma}\label{giamma2}
The function
\[
\varphi(t)=\frac{\pi\, 2^{4/t}\,\Gamma^2\left(\frac12+\frac1t\right)}{t\,\tan\left(\frac{\pi}{t}\right)\,\Gamma^2\left(1+\frac1t\right)}
\]
is increasing and concave on $[3,+\infty)$.
\end{lemma}
\begin{proof}
The first step consists in writing $\varphi(t)$ as a product of Gamma functions. This can be done by recalling the
well known identity $z\Gamma(z) = \Gamma(1+z)$, together with Euler's reflection formula
\[
\Gamma(z)\Gamma(1-z) = \frac{\pi}{\sin(\pi z)},\qquad \Gamma\left(\frac12 + z\right)\Gamma\left(\frac12 - z\right) = \frac{\pi}{\cos(\pi z)},
\]
and Legendre's duplication formula
\[
\Gamma\left(2 z\right)=\frac{2^{2z}}{2\sqrt{\pi}}\Gamma\left( z\right)\Gamma\left(\frac12 + z\right)
\]
(see \cite{Gamma}). Writing $\alpha = 1/t \in(0,1/3]$, we obtain
\[
\begin{split}
\frac{1}{\pi}\varphi\left(\frac{1}{\alpha}\right) &= \frac{2^{2\alpha}\,\Gamma\left(\frac12+\alpha\right)\cdot\cos(\pi \alpha)\Gamma\left(\frac12+\alpha\right)}{
2^{-2\alpha}\,\Gamma\left(1+\alpha\right)\cdot\sin(\pi \alpha)\Gamma\left(\alpha\right)} = \frac{2\sqrt{\pi}\,\frac{\Gamma\left(2\alpha\right)}{\Gamma\left(\alpha\right)}
\cdot\frac{\pi}{\Gamma\left(\frac12-\alpha\right)}}{
2^{-2\alpha}\,\Gamma\left(1+\alpha\right)\cdot\frac{\pi}{\Gamma\left(1-\alpha\right)}} \\
&= \frac{2\sqrt{\pi}}{2^{-2\alpha}\Gamma\left(\frac12-\alpha\right)}\cdot\frac{\Gamma\left(2\alpha\right)}{\Gamma\left(\alpha\right)}
\cdot\frac{\Gamma\left(1-\alpha\right)}{\Gamma\left(1+\alpha\right)} = - \frac{\Gamma\left(-\alpha\right)}{\Gamma\left(-2\alpha\right)}\cdot\frac{\Gamma\left(2\alpha\right)}{\Gamma\left(\alpha\right)}
\cdot\frac{\Gamma\left(-\alpha\right)}{\Gamma\left(\alpha\right)}\\
&= \exp\left(g(\alpha)\right),
\end{split}
\]
where
\[
g(\alpha) =  \log \left[-\frac{\Gamma\left(2\alpha\right)\Gamma^2\left(-\alpha\right)}{\Gamma\left(-2\alpha\right)\Gamma^2\left(\alpha\right)}\right].
\]
Since
\[
\frac{1}{\pi}\varphi'\left(t\right)
= \varphi\left(t\right)\cdot\left[-\frac{1}{t^2}g'\left(\frac{1}{t}\right)\right],
\quad
\frac{1}{\pi}\varphi''\left(t\right)
= \varphi\left(t\right)\cdot\left[\frac{1}{t^4}g'\left(\frac{1}{t}\right)^2 + \frac{2}{t^3}g'\left(\frac{1}{t}\right) + \frac{1}{t^4}g''\left(\frac{1}{t}\right)\right],
\]
the lemma will follow once we show that, whenever $0<\alpha\leq 1/3$,
\begin{equation}\label{eq:gamma1}
g'(\alpha)<0,\qquad \alpha^4g'\left(\alpha\right)^2 + 2\alpha^3g'\left(\alpha\right) + \alpha^4g''\left(\alpha\right)<0.
\end{equation}
To this aim we notice that the derivatives of $g$ can be evaluated in terms of the digamma function $\psi$, given by
\[
\psi(z) := \frac{d}{dz} \log\Gamma(z) = -\gamma -\frac{1}{z} -\sum_{n = 1}^{\infty}\left(\frac{1}{n+z}-\frac{1}{n}\right).
\]
Since the series above (and that of its derivatives) converges uniformly in $\{|z| \leq 3/4\}$ we obtain
\[
\begin{split}
g'(\alpha) &= 2\left[ \psi(2\alpha) + \psi(-2\alpha) - \psi(\alpha) - \psi(-\alpha)\right] \\ &=
2\sum_{n = 1}^{\infty}\left(\frac{1}{n+\alpha} + \frac{1}{n-\alpha} - \frac{1}{n+2\alpha} - \frac{1}{n-2\alpha}\right)
= - \sum_{n = 1}^{\infty}\frac{12\alpha^2 n}{(n^2-\alpha^2)(n^2-4\alpha^2)},
\end{split}
\]
and the first inequality in \eqref{eq:gamma1} follows. Furthermore, let us introduce the function
\[
h(x) = \frac{12\alpha^2 x}{(x^2-\alpha^2)(x^2-4\alpha^2)}.
\]
Through direct inspection we infer that, if $\alpha$ is in the desired interval and $x\geq1$, it holds
\[
h(x)>0, \qquad h'(x) =  \frac{12\alpha^2 (-3x^4 + 5\alpha^2x^2 + 4\alpha^4)}{(x^2-\alpha^2)^2(x^2-4\alpha^2)^2} \leq \frac{12\alpha^2 (-3 + 5\alpha^2 + 4\alpha^4)}{(x^2-\alpha^2)^2(x^2-4\alpha^2)^2} <0;
\]
this allows to use the elementary inequality
\[
h(1) + h(2) < \sum_{n=1}^\infty h(n) < h(1) + \int _1^ \infty h(x)\,dx
\]
in order to estimate
\begin{equation}\label{eq:digamma1}
\begin{split}
- g'(\alpha) & <   \frac{12\alpha^2 }{(1-\alpha^2)(1-4\alpha^2)} + 2\left[\log\frac{x^2-4\alpha^2}{x^2-\alpha^2}\right]_{x=1}^{x=\infty} \leq 12\alpha^2 \cdot \frac98 \cdot \frac95 + 2\log\frac{1-\alpha^2}{1-4\alpha^2} \\
&\leq \frac{243}{10} \alpha^2 + 2 \left( \frac{1-\alpha^2}{1-4\alpha^2} - 1\right) \leq \frac{243}{10} \alpha^2 + 6\alpha^2 \cdot \frac95 < 36 \alpha^2
\end{split}
\end{equation}
and
\begin{equation}\label{eq:digamma2}
- g'(\alpha) > \frac{12\alpha^2 }{(1-\alpha^2)(1-4\alpha^2)} + \frac{24\alpha^2 }{(4-\alpha^2)(4-4\alpha^2)} > 12\alpha^2 + \frac32 \alpha^2 > 13\alpha^2.
\end{equation}
Analogously, we have
\begin{equation}\label{eq:digamma3}
\begin{split}
-g''(\alpha) &=  \sum_{n = 1}^{\infty}\frac{\partial}{\partial\alpha}\frac{12\alpha^2 n}{(n^2-\alpha^2)(n^2-4\alpha^2)} =  12\sum_{n = 1}^{\infty}\frac{2\alpha n^5 - 8\alpha^5 n}{(n^2-\alpha^2)^2(n^2-4\alpha^2)^2}\\
 &> 12\frac{2\alpha  - 8\alpha^5}{(1-\alpha^2)^2(1-4\alpha^2)^2} \ge 12 \left(2-\frac{8}{81}\right)\alpha > 22\alpha
\end{split}
\end{equation}
(also this series has positive terms).
Taking into account equations \eqref{eq:digamma1}, \eqref{eq:digamma2} and \eqref{eq:digamma3} we finally deduce
\[
\alpha^4g'\left(\alpha\right)^2 + 2\alpha^3g'\left(\alpha\right) + \alpha^4g''\left(\alpha\right) < 36^2 \alpha^8 - 26 \alpha^5 - 22 \alpha^5 \leq
\alpha^5\left(36^2 \cdot\frac{1}{27} - 48\right) = 0,
\]
and also the second inequality in \eqref{eq:gamma1} follows.
\end{proof}

\subsection*{Funding} Work partially supported  by the PRIN-2015KB9WPT Grant:
``Variational methods, with applications to problems in mathematical physics and geometry'',
by the ERC Advanced Grant 2013 n. 339958:
``Complex Patterns for Strongly Interacting Dynamical Systems - COMPAT'',
and by the INDAM-GNAMPA group.


\begin{thebibliography}{10}

\bibitem{BoVe}
{B.} Bogosel and {B.} Velichkov, \emph{A multiphase shape optimization problem
  for eigenvalues: qualitative study and numerical results}, SIAM J. Numer.
  Anal. \textbf{54} (2016), no.~1, 210--241.

\bibitem{BNHV}
{V.} Bonnaillie-No{\"e}l, {B.} Helffer, and {G.} Vial, \emph{Numerical
  simulations for nodal domains and spectral minimal partitions}, ESAIM Control
  Optim. Calc. Var. \textbf{16} (2010), no.~1, 221--246.

\bibitem{BuBuHe}
{D.} Bucur, {G.} Buttazzo, and {A.} Henrot, \emph{Existence results for some
  optimal partition problems}, Adv. Math. Sci. Appl. \textbf{8} (1998), no.~2,
  571--579.

\bibitem{BF16}
{D.} Bucur and {I.} Fragal{\`a}, \emph{A {F}aber--{K}rahn {I}nequality for the
  {C}heeger {C}onstant of {$N$}-gons}, J. Geom. Anal. \textbf{26} (2016),
  no.~1, 88--117.

\bibitem{CaffLin}
{L. A.} Caffarelli and {F. H.} Lin, \emph{An optimal partition problem for
  eigenvalues}, J. Sci. Comput. \textbf{31} (2007), no.~1-2, 5--18.

\bibitem{Car15}
{M.} Caroccia, \emph{Cheeger $N$-clusters}, J. Convex Anal. \textbf{22} (2015),
  no.~4, 1207--1213.

\bibitem{CarQuantitative}
M.~Caroccia and R.~Neumayer, \emph{A note on the stability of the {C}heeger
  constant of {$N$}-gons}, J. Convex Anal. \textbf{22} (2015), no.~4,
  1207--1213. \MR{3436708}

\bibitem{Ch}
{J.} Cheeger, \emph{A lower bound for the smallest eigenvalue of the
  {L}aplacian}, Problems in analysis ({P}apers dedicated to {S}alomon
  {B}ochner, 1969), Princeton Univ. Press, Princeton, N. J., 1970,
  pp.~195--199.

\bibitem{Cocu}
{A.} Colesanti and {P.} Cuoghi, \emph{The {B}runn-{M}inkowski inequality for
  the {$n$}-dimensional logarithmic capacity of convex bodies}, Potential Anal.
  \textbf{22} (2005), no.~3, 289--304.

\bibitem{CTV03}
{M.} Conti, {S.} Terracini, and {G.} Verzini, \emph{An optimal partition
  problem related to nonlinear eigenvalues}, J. Funct. Anal. \textbf{198}
  (2003), no.~1, 160--196.

\bibitem{CTV05}
{M.} Conti, {S.} Terracini, and {G.} Verzini, \emph{On a class of optimal
  partition problems related to the {F}u\v c\'\i k spectrum and to the
  monotonicity formulae}, Calc. Var. Partial Differential Equations \textbf{22}
  (2005), no.~1, 45--72.

\bibitem{Gamma}
{A.} Erd{\'e}lyi, {W.} Magnus, {F.} Oberhettinger, and {F.G.} Tricomi, Higher
  transcendental functions. {V}ols. {I}, {II}, McGraw-Hill Book Company, Inc.,
  New York-Toronto-London, 1953, Based, in part, on notes left by Harry
  Bateman.

\bibitem{FT64}
L.~Fejes~T{\'o}th, Regular figures, A Pergamon Press Book, The Macmillan Co.,
  New York, 1964.

\bibitem{Hales}
{T. C.} Hales, \emph{The honeycomb conjecture}, Discrete Comput. Geom.
  \textbf{25} (2001), no.~1, 1--22.

\bibitem{He07}
{B.} Helffer, \emph{Domaines nodaux et partitions spectrales minimales
  (d'apr\`es {B}. {H}elffer, {T}. {H}offmann-{O}stenhof et {S}. {T}erracini)},
  S\'eminaire: \'{E}quations aux {D}\'eriv\'ees {P}artielles. 2006--2007,
  S\'emin. \'Equ. D\'eriv. Partielles, \'Ecole Polytech., Palaiseau, 2007,
  pp.~Exp. No. VIII, 23.

\bibitem{He10}
{B.} Helffer, \emph{On spectral minimal partitions: a survey}, Milan J. Math.
  \textbf{78} (2010), no.~2, 575--590.

\bibitem{HeHoTe}
{B.} Helffer, {T.} Hoffmann-Ostenhof, and {S.} Terracini, \emph{Nodal domains
  and spectral minimal partitions}, Ann. Inst. H. Poincar\'e Anal. Non
  Lin\'eaire \textbf{26} (2009), no.~1, 101--138.

\bibitem{H06}
{A.} Henrot, Extremum problems for eigenvalues of elliptic operators, Frontiers
  in Mathematics, Birkh\"auser Verlag, Basel, 2006.

\bibitem{rsj16}
{R.} Jones, \emph{Computing ultra-precise eigenvalues of the Laplacian within
  polygons}, Arxiv Preprint, arXiv:1602.08636 (2016).

\bibitem{KaLR}
{B.} Kawohl and {T.} Lachand-Robert, \emph{Characterization of {C}heeger sets
  for convex subsets of the plane}, Pacific J. Math. \textbf{225} (2006),
  no.~1, 103--118.

\bibitem{Landkof}
{N. S.} Landkof, Foundations of modern potential theory, Springer-Verlag, New
  York-Heidelberg, 1972, Translated from the Russian by A. P. Doohovskoy, Die
  Grundlehren der mathematischen Wissenschaften, Band 180.

\bibitem{Leo}
{G. P.} Leonardi, \emph{An overview on the Cheeger problem}, Pratelli, {A.},
  Leugering, {G.} (eds.) New trends in shape optimization., International
  Series of Numerical Mathematics, Springer (Switzerland), vol. 166, 2016,
  pp.~117--139.

\bibitem{Morgan}
{F.} Morgan, \emph{The hexagonal honeycomb conjecture}, Trans. Amer. Math. Soc.
  \textbf{351} (1999), no.~5, 1753--1763.

\bibitem{Pa}
{E.} Parini, \emph{An introduction to the {C}heeger problem}, Surv. Math. Appl.
  \textbf{6} (2011), 9--21.

\bibitem{Ramos}
{M.} Ramos, {H.} Tavares, and {S.} Terracini, \emph{Extremality conditions and
  regularity of solutions to optimal partition problems involving {L}aplacian
  eigenvalues}, Arch. Ration. Mech. Anal. \textbf{220} (2016).

\bibitem{SoZa}
{A.Y.} Solynin and {V. A.} Zalgaller, \emph{An isoperimetric inequality for
  logarithmic capacity of polygons}, Ann. of Math. (2) \textbf{159} (2004),
  no.~1, 277--303.

\end{thebibliography}

\def\cprime{$'$}
\providecommand{\bysame}{\leavevmode\hbox to3em{\hrulefill}\thinspace}
\providecommand{\MR}{\relax\ifhmode\unskip\space\fi MR }
\providecommand{\MRhref}[2]{%
  \href{http://www.ams.org/mathscinet-getitem?mr=#1}{#2}
}
\providecommand{\href}[2]{#2}

\end{document}